\documentclass[12pt,titlepage]{amsart}
\usepackage{setspace}
\usepackage{amsmath,amsthm,amssymb,amsfonts}
\usepackage{hyperref}
\hypersetup{colorlinks=true}
\numberwithin{equation}{section}
\makeindex
\newcommand{\tr}{\mathrm{tr}}

\newcommand{\supp}{\mathrm{supp}}

\newcommand{\Nm}{\mathrm{Nm}}
\newcommand\xappa\kappa
\newcommand\yota\iota
\newcounter{consta}

\newcounter{constb}

\newcounter{constc}[section]

\newcommand{\vol}{\operatorname{vol}}

\newcommand{\largewedge}{\mbox{\Large $\wedge$}}

\newcommand{\dist}{\mathrm{dist}}

\DeclareFontFamily{OT1}{rsfs}{}
\DeclareFontShape{OT1}{rsfs}{n}{it}{<-> rsfs10}{}
\DeclareMathAlphabet{\mathscr}{OT1}{rsfs}{n}{it}
\swapnumbers
\newtheorem{thm}{Theorem}[section]
\newtheorem{lem}[thm]{Lemma}

\newtheorem{prop}[thm]{Proposition}
\newtheorem*{lem*}{Lemma}
\newtheorem*{thm*}{Theorem}
\newtheorem*{conj*}{Conjecture}
\newtheorem*{prop*}{Proposition}
\newtheorem{defn*}{Definition}

\newtheorem{ex}[thm]{Example}
\newtheorem{cor}[thm]{Corollary}

\theoremstyle{definition}
\newtheorem{defn}[thm]{Definition}
\theoremstyle{remark}
\newtheorem{rem}[thm]{Remark}

\newtheorem*{obs*}{Observation}
\newtheorem*{rem*}{Remark}
\theoremstyle{definition}\newtheorem*{acknowledgments}{Acknowledgments}

\begin{document}
\title[Disjointness of nilflows from horospherical flows]{Quantitative disjointness of nilflows from horospherical flows}
\author{Asaf Katz}
\address{Department of Mathematics, University of Chicago, Chicago, Illinois 60637, USA,}
\email{asafk@math.uchicago.edu}

\date{}

\begin{abstract} We prove a quantitative variant of a disjointness theorem of nilflows from horospherical flows following a technique of Venkatesh, combined with the structural theorems for nilflows by Green, Tao and Ziegler.
\end{abstract}
\maketitle
\section{Introduction}
In a landmark paper~\cite{furstenberg67}, H. Furstenberg introduced the notion of joinings of two dynamical systems and the concept of disjoint dynamical systems.
Ever since, this property has played a major role in the field of dynamics, leading to many fundamental results.

In his paper, Furstenberg proved the following characterization of a weakly-mixing dynamical system:
\begin{thm}[Furstenberg \cite{furstenberg67}, Theorem~$1.4$]
A dynamical system $(X,\beta,\mu,T)$ is \emph{weakly-mixing} if and only if $(X,T)$ is disjoint from any Kronecker system.
\end{thm}
Namely, given any compact abelian group $A$, equipped with the action of $A$ on itself by left-translation $R_{a}$, the only joining between $(X,T$) and $(A,R_{a})$ is the trivial joining given by the product measure on the product dynamical system $(X\times A, T\times R_{a})$.
The Wiener-Wintner ergodic theorem readily follows from Furstenberg's theorem.
In \cite{bourgain90}, J. Bourgain derived a strengthening of Furstenberg's disjointness theorem, which amounts to the \emph{uniform Wiener-Wintner theorem}. The proof is along the lines of Furstenberg's but utilizing the Van-der-Corput trick to show uniformity over various Kronecker systems.
In a recent work~\cite{venkatesh10}, A. Venkatesh gave a \emph{quantitative} statement of the uniform Wiener-Wintner theorem for the case of the horocyclic flow on compact homogeneous spaces of $SL_{2}(\mathbb{R})$ (and in principal, his proof works also for non-compact spaces as well, by allowing the decay rate to depend on Diophantine properties of the origin point of the orbit).
Venkatesh's proof follows Bourgain's, but crucially uses a quantitative version of the Dani-Smillie theorem and quantitative estimates regarding decay of matrix coefficients in order to deduce the required effective estimate. Related work has been discussed in~\cite{Flaminio2016,tanis2015,sarnakubis}.

In this work, we extend Venkatesh's method in order to prove disjointness between general nilflows and horospherical flows.
Nilflows are generalizations of Kronecker systems which we define bellow, following the conventions of Green-Tao-Ziegler~\cite{greentaoziegler12} and the related work of Green-Tao~\cite{greentao12} and Leibman~\cite{leibman2002}:
\begin{defn}
\begin{itemize}
	\item A \emph{nilmanifold} $Y$ is a homogeneous space $Y=~N/\Lambda$ where $N$ is a nilpotent Lie group and $\Lambda$ is a lattice contained in $N$.
	\item For a subgroup $H\leq N$, we have a natural $H$-action on a nilmanifold $Y=N/\Lambda$ by left-translations. An \emph{$H$-nilflow} is the dynamical system $(Y,H)$.
	\item A ($H$-)\emph{nilsequence} of degree $\leq \dim N$ is the set of samplings of a Lipschitz continuous function $f:Y \to \mathbb{C}$ along a particular $H$-orbit $H.y\subset Y$ for some $y\in Y$.
	\item A ($H$-)\emph{nilcharacter} of degree $\leq \dim N$ is a nilsequence where the sampling function $f$ satisfies that $\left\lVert f\right\rVert_{\infty} =1$ and there exists some character $\chi$ in the dual group to $Z(N)/Z(N)\cap \Lambda$ such that transformation rule $f(g.y)~=~\chi(g)f(y)$ holds for every $y\in~N/\Lambda, g\in Z(N)$.
\end{itemize}
Nilcharacters take the role of characters in the analysis of nilflows on nilpotent groups which are not abelian.
We note the following basic properties of nilcharacters:
\begin{itemize}
	\item [Approximation] Every nilsequence of degree $\leq \dim N$ can be uniformly approximated by a linear combination of nilcharacters of degree $\leq \dim N$. More precisely, given $\varepsilon>0$, one may approximate a nilsequence of degree $\leq \dim N$ by a linear combination of $O\left(1/\varepsilon^{\dim N}\right)$ nilcharacters of degree $\leq \dim N$ and their coefficients are bounded by $\|f\|_{\infty}$, where $f$ is the sampling function of the nilsequence up to an error of $O(\varepsilon)$. This assertion follows from a Fejer kernel computation and a partition of unity argument (c.f.~\cite[Lemma~$3.7$]{greentao12},\cite[Lemma~$E.5$, \S6]{greentaoziegler12}). \\
	\item [Differentiation] Given a nilcharacater of degree $\leq\dim N$ defined as $f(h.y)$ for some $f:N/\Lambda \to \mathbb{C}$ along the orbit $H.y \subset N/\Lambda$ and fixing any $k\in H$ we have that the ``differentiated'' product $f(k.h.y)\cdot~\overline{f(h.y)}$ is a \emph{nilsequence of degree $\leq \dim N -1$}. This is proved in \cite[Lemma~$E.8.(iv),E.7$]{greentaoziegler12},\cite[Lemma~$1.6.13$]{tao2012}.
\end{itemize}
\end{defn}

The main examples to keep in mind are nilflows which are realized on abelian groups which amount to flows over quotient spaces of $\mathbb{R}^{d}$ and nilflows which are realized on meta-abelian groups which amount to flows over homogenous spaces of the Heisenberg group.
In the first case, $\mathbb{Z}$-nilcharacters amount to linear characters such as $e(n\cdot \alpha)$ while in the second case, $\mathbb{Z}$-nilcharacters amount to ``quadratic characters'' defined by quadratic bracket polynomials, such as $e(n^{2}\cdot\alpha)$ and $e(n\alpha\cdot \left\{n\beta\right\})$.

Nilflows play a fundamental role in the work by Furstenberg and Weiss~\cite{FuWe96} about multiple recurrence properties of dynamical systems, and ever since had a substantial role in the proofs of many non-conventional recurrence theorems including the Green-Tao theorem regarding existence of arithmetic progressions inside the set of prime numbers. We refer the reader to \cite[Chapter $1$]{tao2012} for a general overview of the theory of nilflows.

We now introduce the required definitions from the theory of unipotent flows over semisimple Lie groups.
\begin{defn}
Let $G$ be a real Lie group, we say that a subgroup $H\leq G$ is \emph{horospherical}, if there exist an element $a\in G$ such that $H=\left\{g\in G \mid a^{n}ga^{-n} \to e, n\to -\infty \right\}$. Equivalently such a subgroup $H$ is the unipotent radical $Rad_{U}(P)$ of a proper parabolic subgroup $P$ of $G$.
\end{defn}
Such groups are nilpotent, connected and simply-connected.
The main example to keep in mind is the unipotent subgroup composed of upper-triangular matrices with ones in the diagonal entries in $SL_{2}(\mathbb{R})$, which can be easily observed to be horospherical by picking $a$ to be a non-trivial diagonal matrix with $a_{1,1}>a_{2,2}$.
Horospherical subgroup $H\leq G$ is called minimal-horospherical if it does not strictly contain another horospherical subgroup (equivalently, $H$ is the unipotent radical of a \emph{maximal parabolic} subgroup of $G$).
Given a one-parameter semi-group $A=\left\{ a_{t} \mid t\geq 0\right\} \subset G$ consisting of semi-simple elements, we will say that $H$ is horospherical with respect to $A$ if $H=\left\{g\in G \mid a_{n}ga_{-n} \to e, n\to -\infty \right\}$.
Moreover, in the case $G/\Gamma$ is non-compact, we will assume that $G$ is defined over $\mathbb{Q}$ and $\Gamma$ is an arithmetic lattice, and from now on we fix a $\mathbb{Q}$-rational proper embedding $i:G/\Gamma~\hookrightarrow~SL_{k}(\mathbb{R})/SL_{k}(\mathbb{Z})$.

For any horospherical subgroup of a semi-simple Lie group $H\leq G$ with an associeted one-parameter expanding semi-group $A=\left<a_{t}\right>$, we define a family of subsets $B^{H}_{R}\subset H$ in the following manner
\begin{equation*}
    B^{H}_{R}=a_{\log R}B^{H}_{1}a_{-\log R},
\end{equation*}
where $B^{H}_{1}\subset H$ is the ball of radius $1$ with respect to a right-invariant Riemannian metric defined over $G$, by the Killing form, defined with respect to some fixed Cartan subgroup of $G$ which contains $A$ as a semi-group of this positive Cartan subgroup.
We note that under this choice of the Cartan subgroup, the restriction of the Killing form over $Lie(H)$ is \emph{definite}, giving rise to a norm on $Lie(H)$.

The family of sets $\left\{B^{H}_{R}\right\}$ is a F\o lner sequence for $H$ (as $H$ is nilpotent).
We define a number $d_{H}$ to be $d_{H}=\tr\left(Ad_{a_{1}\vert_{Lie(H)}}\right)$.
The number $d_{H}$ is related to $\vol_{H}\left( B_{R}^{H}\right)$ by using the formula for the modular function of a parabolic group $P$ which contains $H$ as its unipotent radical (c.f. \cite[Propositions~$8.27$,$8.44$,$8.45$]{knapp2013lie}), in the following manner:
\begin{equation}\label{eq:ball-vol}
\vol_{H}\left(B_{R}^{H}\right) = \vol_{H}\left(a_{\log R}B_{1}^{H}a_{-\log R}\right) = R^{d_{H}}\cdot \vol_{H}\left(B_{1}^{H}\right).
\end{equation}

\begin{defn}
For a smooth function $f:X\to \mathbb{C}$, fixing a basis $L=\left\{X_{i}\right\}$ for the Lie algebra $Lie(G)$ we define the \emph{order $K$ Sobolev norm} as
\begin{equation*}
    Sob_{K}(f)=\max_{0\leq \ell \leq K}\left\{\lVert X_{i_{1}}\cdots X_{i_{\ell}}.f \rVert_{\infty} \mid X_{i_j} \in L\right\},
\end{equation*}
where the vector fields $\left\{X_{i}\right\}$ acts as a derivation of the smooth function $f$.
\end{defn}

\begin{defn}\label{def:quant-equi}
Let $G$ be a semi-simple linear group without compact factors, and let $\Gamma\leq G$ be a lattice, and denote $X=G/\Gamma$.
Assume that $H\leq G$ is an horospherical subgroup with associated expanding semi-group $A=\left<a_{t}\right>$, and the associated averaging family of subsets $\left\{ B^{H}_{R}\right\}\subset~H$.
Let $x\in X$ be an $H$-generic point, namely $\overline{H.x}=X$.
We say that the orbit $H.x$ \emph{equidistributes with a polynomial rate $\gamma_{\text{Equidistribution}}$} for functions with finite Sobolev norm of order $K$ if fo any smooth and bounded function with vanishing integral having $Sob_{K}(f)<\infty$ and for any $\eta<\gamma_{\text{Equidistribution}}$, the following estimate holds:
\begin{equation*}
    \left\lvert \frac{1}{\vol_{H}\left(B^{H}_{R}\right)}\int_{u\in B^{H}_{R}}f(u.x)du \right\rvert \ll \left\lvert \vol_{H}\left(B^{H}_{R}\right)\right\rvert^{-\eta}\cdot Sob_{K}(f).
\end{equation*}
\end{defn}

Our first main theorem provides a quantitative disjointness statement between nilflows and horospherical orbits which equidistribute in a polynomial rate.
Let $H$ be a fixed real linear connected nilpotent Lie group.
We assume that there exists some linear group $G$ which is semi-simple, with no compact factors and contains $H$ as a horospherical subgroup with respect to some given element $a\in G$.
Moreover, we assume that there exists some linear group $N$ which is nilpotent and contains $H$ as a subgroup.
We fix two lattices $\Gamma \leq G$ and $\Lambda \leq N$ and denote by $X=G/\Gamma$ and $Y=N/\Lambda$ the associated homogeneous spaces. In the case where $\Gamma$ is non-uniform, we will also assume that $G$ is defined over $\mathbb{Q}$ and $\Gamma$ is arithmetic.
$H$ acts by left-translations on both $X$ and $Y$.
We note here that by the Howe-Moore theorem, the dynamical system $(X,H)$ is mixing.

The dynamical systems $(X,H)$ and $(Y,H)$ are quantitatively disjoint in the following sense:
\begin{thm}\label{thm:effective-disjointness}
\sloppy For any $x\in X$ which is $H$-generic with polynomial equidistribution rate of $\gamma_{\text{Equidistribution}}$ with respect to $\left\{B^{H}_{R}\right\}$ for functions with finite Sobolev norm of order $K$, there exists a number  $\gamma_{\text{Disjointness}}~=~\gamma_{\text{Disjointness}}(\Gamma,\gamma_{\text{Equidistribution}},\dim N)>~0$ and $K'=K'(N,K)>0$ such that for any nilcharacter $\psi : Y \to \mathbb{C}$ of degree less or equal than $\dim N$ and any $f:X\to\mathbb{C}$ which is smooth, bounded, with vanishing integral and of finite $K'$ Sobolev norm, the following estimate holds:
\begin{equation}\label{eq:quant-disjoint}
\left\lvert \frac{1}{\vol_{H}\left(B^{H}_{R}\right)}\int_{u \in B^{H}_{R}}\psi(u.y)f(u.x)du\right\rvert  \ll_{f} \left\lvert \vol_{H}\left(B^{H}_{R}\right)\right\rvert^{-\eta},
\end{equation}
for any $\eta<\gamma_{\text{Disjointness}}$.
One may take
\begin{equation}
\gamma_{\text{Disjointness}}<\frac{1}{\left(2\cdot\dim H +2\right)^{\dim N}}\cdot\left(\frac{2M}{2M+1}\right)^{\dim N}\min\left\{\gamma_{\text{Equidistribution}},\frac{s}{d_{H}\cdot(2d_{H}+3s)}\right\},
\label{eq:gamma-def}
\end{equation}
where $s>0$ is a bound for the decay rates of the matrix coefficient $\left<h.f,f \right>_{L^{2}(X,\mu)}$ for $h\in H$, and $M$ is the maximum between $K'$ and the order of the Sobolev norm used in the mixing estimate of Theorem~\ref{thm:HC-mixing}.
\end{thm}

As an example for an application, we give the following corollary, which establishes quantitative cancellation in horocyclic averages over compact surfaces twisted by quadratic characters:
\begin{cor}
Let $G=SL_{2}(\mathbb{R})$, $\Gamma \leq G$ be a uniform lattice.
Then for any $f:G/\Gamma \to \mathbb{R}$ which is smooth, compactly-supported and of vanishing integral with finite Sobolev norm of order $3$, any $x\in G/\Gamma$ and any $\alpha \in \mathbb{R}$, the following estimate holds:
\begin{equation}
\left\lvert \frac{1}{2R}\int_{t=-R}^{R}e\left(\alpha t^{2}/2-\left\{\sqrt\alpha t \right\}\sqrt\alpha t\right)f(u_{t}.x)dt \right\rvert \leq_{f} R^{-\gamma},
\label{eq:generalized-poly-venkatesh}
\end{equation}
for any $\gamma<~\frac{6^{3}}{5^{3}}\cdot\frac{\Re(s_{1})}{2+3\cdot\Re(s_1)}$, where $\lambda_{1}=s_{1}(1-s_{1})$ the first non-trivial Laplacian eigenvalue.
In the case where $\Re(s_1)=1$, we may choose any $\gamma<\frac{6^3}{5^6 \cdot 7}\approx 0.0019\dots$
\end{cor}
\begin{proof}
The corollary follows from Theorem~\ref{thm:effective-disjointness} by embedding the additive group of $\mathbb{R}$, $\mathbb{G}_{a}(\mathbb{R})$ in two ways:
First, embed $\mathbb{G}_{a}(\mathbb{R})\hookrightarrow SL_{2}(\mathbb{R})$ by identifying $\mathbb{G}_{a}(\mathbb{R})$ with the upper-triangular unipotent subgroup of $SL_{2}(\mathbb{R})$.
Second, embed $\mathbb{G}_{a}(\mathbb{R})\hookrightarrow N(\mathbb{R})$, where $N(\mathbb{R})$ is the Heisenberg group, identified with its copy inside $SL_{3}(\mathbb{R})$ by $$t \mapsto \exp\left(\sqrt{\alpha}\cdot t\cdot\left(\begin{smallmatrix}
	0 & 1 & 0 \\
	0 & 0 & 1 \\
	0 & 0 & 0
\end{smallmatrix}\right)  \right).$$
It follows that $e\left(\alpha t^{2}/2-\left\{\sqrt\alpha t \right\}\sqrt\alpha t\right)$ can be realized as an $\mathbb{R}$-nilcharacter on a $\mathbb{G}_{a}(\mathbb{R})$-orbit in the nilmanifold $N(\mathbb{R})/N(\mathbb{Z})$, by picking the origin point $y=\left(\begin{smallmatrix}
	1 & 0 & 0 \\
	0 & 1 & 0 \\
	0 & 0 & 1
\end{smallmatrix}\right)N(\mathbb{Z}) \in N(\mathbb{R})/N(\mathbb{Z})$ and the sampling function $F:N(\mathbb{R})/N(\mathbb{Z}) \to \mathbb{C}$ to be
\begin{equation*}
    F\left(\left(\begin{smallmatrix}
	1 & x & y \\
	0 & 1 & z \\
	0 & 0 & 1
\end{smallmatrix}\right) \right) = e(y-\left\{z\right\}\cdot x).
\end{equation*}
We note that this function is indeed $N(\mathbb{Z})$-invariant by doing the following computation:
\begin{equation*}
\left(\begin{smallmatrix}
	1 & x & y \\
	0 & 1 & z \\
	0 & 0 & 1
\end{smallmatrix}\right)\cdot \left(\begin{smallmatrix}
	1 & a & b \\
	0 & 1 & c \\
	0 & 0 & 1
\end{smallmatrix}\right) = \left(\begin{smallmatrix}
	1 & a+x & b+xc+y \\
	0 & 1 & c+z \\
	0 & 0 & 1
\end{smallmatrix}\right),    
\end{equation*}
and hence
\begin{equation*}
\begin{split}
F\left(\left(\begin{smallmatrix}
	1 & x & y \\
	0 & 1 & z \\
	0 & 0 & 1
\end{smallmatrix}\right)\cdot \gamma\right) &= e((b+xc+y)-\left\{c+z \right\}(a+x)),
\end{split}
\end{equation*}
for $\gamma\in N(\mathbb{Z})$ and we have the following equalities modulo $1$:
\begin{equation*}
\begin{split}
    b+xc+y-\left\{ c+z\right\}\cdot(a+x) &\equiv y+xc-\left\{ c+z\right\}\cdot(a+x) \\
    &\equiv y+xc-\left\{ c+z\right\}\cdot(x) \\
    &\equiv y-\left\{z\right\}\cdot x.
\end{split}
\end{equation*}

Using \eqref{eq:quant-disjoint} with the definition of $\gamma$ as given in \eqref{eq:gamma-def}, combined with a sharp equidistribution result proved for this case by M. Burger~\cite[Theorem~$2.C$]{burger1990}, as we have that $d_{H}=1, \dim H=1$, $\dim N=3$ and $M=K'=12$ as we need to use twice the differentiation result resulting in a factor of $4$, we deduce the claimed estimate.
One may recover a disjointness statement for the $\mathbb{R}$-nil-sequence $e(\alpha\cdot t^2)$ in a similar manner by taking a product construction over this nilmanifold, as explained in \cite[Example $5$]{green-tao-quad}.

\end{proof}
The proof of the above mentioned theorem is a quantification of the following qualitative disjointness theorem:
\begin{thm}\label{thm:non-effective}
Let $G$ be a semi-simple Lie group without compact factors. Assume that $H\leq G$ is a closed and connected nilpotent group, and $x\in X=G/\Gamma$ is an element such that the orbit $H.x$ equidistributes in $X$ when sampled over the F\o lner sequence formed by the subsets $\left\{B^{H}_{R}\right\}$.
Then for any nilflow $(N/\Lambda,T)$ admitting an $H$-action and every point $y \in N/\Lambda$ the following holds:  for every bounded Lipschitz continuous functions $f_{1}:~N/\Lambda~\to~\mathbb{C}, f_{2}:G/\Gamma \to \mathbb{C}$ we have:
\begin{equation}\label{eq:non-effective-thm}
\frac{1}{\vol_{H}\left(B^{H}_{R}\right)}\int_{u \in B^{H}_{R}}f_{1}(u.y)f_{2}(u.x)du \to \int_{\overline{H.y}}f_{1}d\overline{H.y}\cdot \int_{G/\Gamma}f_{2}dm,
\end{equation}
where $d\overline{H.y}$ is the normalized Haar measure supported on the \emph{homogenous} orbit closure of the $H$-orbit $\overline{H.y} \subset N/\Lambda$.
\end{thm}
The related joining classification in the case where the group $G$ in Theorem~\ref{thm:non-effective} is nilpotent is due to Lesigne~\cite{lesigne1991} and quantitatively by Green-Tao~~\cite{greentao12}.

While the non-quantitative disjointness theorem, Theorem~\ref{thm:non-effective} follows easily from Ratner's measure classification theorem~\cite[Theorem~$1$]{ratner91} regarding measure classification of measures on homogeneous spaces of general Lie groups which are invariant under actions of subgroups generated by unipotent elements, the proof we give only makes usage of the equidistribution theorem for homogeneous space of semi-simple Lie groups and it is based on analysis of nilsequences developed by Green-Tao-Ziegler~\cite{greentaoziegler12} for characteristic factors of nilflows and follows the lines of the results by Furstenberg, Bourgain and Venkatesh.

Our second main theorem of the paper provides a quantitative version of Dani's horopspherical equidistribution theorem.
In order to state the theorem, we need to define the notion of a diophantine point in the homogeneous space $G/\Gamma$.
\begin{defn}\label{def:alpha-func}
For $m=1,\ldots,k$ and $x\in SL_{k}(\mathbb{R})/SL_{k}(\mathbb{Z})$ which we represent as $x=g\cdot SL_{k}(\mathbb{Z})$, we define the quantity $\alpha_{i}(x)$ as follows:
\begin{equation}\label{eq:alpha-def}
\alpha_{i}(x)=1/\min\left\{\lVert g.v \rVert_{\infty} \mid v\in \largewedge^{i}\left(\mathbb{Z}\right)\setminus\{0\} \right\},
\end{equation}
where the $\infty$-norm of the wedge product is taken with respect to the basis given by $e_{i_1}\wedge\cdots\wedge e_{i_m}$ where $\left\{ e_{i}\right\}$ is the standard basis of $\mathbb{R}^{k}$.
It is easily seen that the definition is independent of the representative matrix $g$ we choose.
Moreover, we define the function $\alpha(x)$ as
\begin{equation*}
\alpha(x)=\max\left\{\alpha_{1}(x),\ldots,\alpha_{k}(x) \right\}.
\end{equation*}
We also define the following subsets of $X=SL_{k}(\mathbb{R})/SL_{k}(\mathbb{Z})$:
\begin{equation*}
\begin{split}
X_{\geq \varepsilon} = \left\{x\in SL_{k}(\mathbb{R})/SL_{k}(\mathbb{Z}) \mid \alpha(x) \leq \varepsilon ^{-1} \right\}, \\
X_{\geq \varepsilon}^{1} = \left\{x\in SL_{k}(\mathbb{R})/SL_{k}(\mathbb{Z}) \mid \alpha_{1}(x) \leq \varepsilon ^{-1} \right\}.
\end{split}
\end{equation*}
In particular, for every $x\in X_{\geq\varepsilon}$ we have that $\lVert x.V \rVert_{\infty} \geq \varepsilon$ for every $V\in \largewedge^{m}(\mathbb{Z})~\setminus~\{0\}$, for every $1\leq m \leq k$.
\end{defn}
We remark that by Mahler's compactness criteria, the family of subsets $\{X_{\geq \varepsilon}^{1}\}$ forms an increasing and exhausting family of compact subsets of $SL_{k}(\mathbb{R})/SL_{k}(\mathbb{Z})$, and we obviously have $X_{\geq\varepsilon}\subset X^{1}_{\geq\varepsilon}$. Using lattice reduction techniques, we also have that $X^{1}_{\geq\varepsilon} \subset X_{\geq C(k)\cdot\varepsilon^{k-1}}$ for some explicit constant $C(k)$.

\begin{defn}\label{def:dioph}
Given $D>0$ a small parameter, diagonalizable one-parameter group $A=\left<a_{t}\right>\leq SL_{k}(\mathbb{R})/SL_{k}(\mathbb{Z})$, a matching subgroup $H\leq~G$ which is contained in the horospherical subgroup associated to the semi-group and a function $\Theta:\mathbb{R}_{\geq 0}\to\mathbb{R}_{\geq 0}$, we say that a point $x\in SL_{k}(\mathbb{R})/SL_{k}(\mathbb{Z})$ is $\Theta$-diophantine (with respect to the given subgroups $A$,$H$ and parameter $D$) if 
\begin{equation*}
    B^{H}_{\Theta(R)}.a_{-\log R}.x \cap X_{\geq R^{-D}} \neq \emptyset, \ \forall R>1.
\end{equation*}
\end{defn}

In view of the uniform Dani-Margulis non-divergence theorem~\cite{dm91}, every point $x\in G/\Gamma$ which is not fixing a rational vector in some projective $\mathbb{Q}$-rational representation of $G$, is $\Theta$-diophantine for some function $\Theta$, for any $D>0$.
Moreover that if $a_{-\log R}.x$ diverges very slowly (with respect to $D$), then trivially the point is $\Theta$-diophantine for $\Theta(R)=1$ say.
We say that $\Theta(T)$ is \emph{polynomially bounded} if there exists some polynomial $p(T)\in\mathbb{Z}[T]$ such that $\Theta(T)\leq \max\{1,p(T)\}$ for every $T\geq 1$.

This definition allows us to control, in a quantitative form, recurrence rates of unipotent orbits to ``almost compact'' sets and is inspired from the definition of the ``diophantine sets'' in~\cite[Definition~$3.6$]{Lindenstrauss2014}, see also \cite[Lemma~$2.6$]{sarnakubis}.
We note here that although images of semisimple and unipotent elements under rational morphisms are indeed semisimple and unipotent respectively~\cite[Theorem~$4.4.4$]{borel2012linear}, this definition \emph{depends on the actual rational embedding which is used}.
We define the \emph{injectivity radius} at a point $x\in G/\Gamma$ as 
\begin{equation*}
InjRad(x)=\sup_{r>0} \left\{x \to g.x \text{ is injective for all }g\in Ball_{G}(r) \right\}
\end{equation*}
where $Ball_{G}(r)$ is the ball of radius $r$ with respect to this metric.
By~\cite[Proposition~$3.5$]{kleinbock-margulis1}, we have the following relationship between injectivity radius and the family $\{X_{\geq \varepsilon}^{1}\}$:
\begin{equation*}
InjRad(x) \geq C \cdot \varepsilon ^{k} \ \ \forall x\in X_{\geq \varepsilon}^{1}.
\end{equation*}

Now we may state a quantitative version of Dani's horospherical equidistribution theorem.
\begin{thm}\label{thm:effective-equi}
Let $G$ be a semi-simple linear group without compact factors, let $\Gamma\leq G$ be a lattice, and denote $X=G/\Gamma$.
Moreover, in the case where $\Gamma$ is non-uniform, we further assume that $G$ is defined over $\mathbb{Q}$ and $\Gamma$ being arithmetic.

Assume that $H\leq G$ is a horospherical subgroup with associated expanding semi-group $A=\left<a_{t}\right>$, and the associated averaging family of subsets $\left\{ B^{H}_{R}\right\}=\left\{a_{\log R}B^{H}_{1}a_{-\log R} \right\}\subset H$.

There is a parameter $D=~D(\Gamma)$ such that if $x$ is $H$-generic, namely $\overline{H.x}=X$ and $\Theta$-diophantine with respect to $A$,$H$,$D$ with $\Theta$ being polynomially bounded then there exists $\gamma_{\text{Equidistribution}}=~\gamma_{\text{Equidistribution}}(\Gamma,\Theta)>~0$ and $K>0$ such that for any smooth and bounded function $f:X\to~\mathbb{C}$ with bounded Sobolev norm of order $K$ and vanishing integral, the following estimate holds:
\begin{equation}\label{eq:quant-equi}
    \left\lvert \frac{1}{\vol_{H}\left(B^{H}_{R}\right)}\int_{u\in B^{H}_{R}}f(u.x)du \right\rvert \ll_{f} \left\lvert \vol_{H}\left(B^{H}_{R}\right)\right\rvert^{-\eta},
\end{equation}
for any $\eta<\gamma_{\text{Equidistribution}}(\Gamma,\Theta).$
\end{thm}
Our proof of this theorem, which is based upon quantitative mixing estimates, yields an actual estimate for $\gamma_{\text{Equidistribution}}$.
In the case where $\Gamma$ is a \emph{uniform lattice} and $H$ is an abelian horospherical group we have that
\begin{equation*}
    \gamma_{\text{Equidistribution}}<\frac{2s}{\dim G +2},
\end{equation*}
where $s>0$ is a bound for the decay rate of the matrix coefficient $\left<h.f,f\right>_{L^{2}(X,\mu)}$ for $h\in H$ in the $G$-representation $G\curvearrowright L^{2}_{0}(X,\mu)$, see Theorem~\ref{thm:HC-mixing}.
Recent work recovering a similar result (in a greater generality with respect to function spaces) was done by  McAdam~\cite{mcadam}, where the author also recoveres a quantitative disjointness statement for multiparameter abelian actions and derives a corollary a sparse equidistribution statement based on sieving methods, in the spirit of Venkatesh. The main contributions of the current paper are handling the non-abelian group action case in Theorem~\ref{thm:effective-equi} and generalizing Venkatesh's result to joining with general nilflows, which require the introduction of machinery which originates in higher Fourier analysis.

\begin{rem}\label{rem:sobolev}
In all the quantitative theorems we prove, we assume that $f$ is a bounded function. It is of interest to prove such quantitative estimates for functions which have a specific growth type at the cusp as well, see for example the definition of Sobolev norms of Str\"ombergsson~\cite[Equation~(3), Theorems~$1,2$]{STR13}. Our treatment does recover such a bound (for a class of functions with sufficiently slow growth at the cusp), but we haven't emphasized this in our computations.
\end{rem}



\subsection*{Structure of the paper}
In section~\S\ref{sec:non-effective} we provide the proof of Theorem~\ref{thm:non-effective}.
In section~\S\ref{sec:effective} we prove the quantitative horospherical equidistribution theorem~\ref{thm:effective-equi}.
In section~\S\ref{sec:quant-disjointness} we show how to deduce the quantitative disjointness theorem~\ref{thm:effective-disjointness} from the proof of the qualitative theorem~\ref{thm:non-effective} using a quantitative equidistribution statement.

\begin{acknowledgments}
The research was conducted as part of the author's PhD thesis in the Hebrew University of Jerusalem, under the guidance of Prof. Elon Lindenstrauss.
The author also wishes to thank Prof. Tamar Ziegler for helpful conversations regarding nilcharacthers and explaining the inverse theorem of \mbox{Green-Tao-Ziegler} to him.
Part of this work was done while the author was visiting MSRI during the program ''Analytic Number Theory''. The author wishes to thank MSRI and the organizers for providing excellent working conditions.
The research was supported by ERC grant (AdG Grant 267259) and the ISF (grant 891/15).
\end{acknowledgments}
\section{Proof of Theorem~\ref{thm:non-effective}}\label{sec:non-effective}
In order to show disjointness from a given nilflow $\left(N/\Lambda,H\right)$ , it is enough to show that the ergodic averages of a smooth function with compact support and vanishing integral along a F\o lner sequence, twisted by a nilcharacter defined over the nilflow, converge to $0$.
Now we prove theorem~\ref{thm:non-effective} by induction on the degree of the nilflow $\left(N/\Lambda,T\right)$.
We begin by proving a variant of the Van-der-Corput inequality along F\o lner sequences.
Recall that a sequence $\left\{F_{n}\right\}\subset G$ of a unimodular group $G$ is called a F\o lner sequence if $F_{n}$ is a sequence of measurable subsets of $G$, of finite measure, converging to $G$ which satisfy $\frac{\lvert g.F_{n}\triangle F_{n}\rvert}{\lvert F_{n}\rvert} \to 0$ for any $g\in G$.
We will only consider the case of $G$ being a nilpotent group, hence unimodular group and of polynomial growth, hence there is no ambiguity in the definition of the ratio $\frac{\lvert g.F_{n}\triangle F_{n}\rvert}{\lvert F_{n}\rvert}$.
Moreover, we will only consider monotone F\o lner sequences in this paper.



\begin{lem}
	Assume that $f:X\to \mathbb{C}$ is a Lipschitz continuous and bounded function, where $X$ is some metric space equipped with a continuous $G$-action for some unimodular group $G$. Let $\left\{F_{n}\right\}$ be a F\o lner sequence for $G$ action and suppose we are given a compact subset $B\subset G$ with small doubling, namely $\lvert B^{-1}\cdot B \rvert \leq K\cdot \lvert B \rvert$ for some fixed $K$; then
	\begin{equation}
	\left\lvert \frac{1}{\lvert F_{n} \rvert}\int_{g\in F_{n}}f(g.x)dg \right\rvert \ll \sqrt{\frac{1}{\lvert B \rvert^{2}}\iint_{b_1 , b_2 \in B} \left\lvert \gamma_{f,x}(b_{1},b_{2})\right\rvert db_1 db_2}+
	2\cdot \lVert f \rVert_{\infty}\cdot \left(\frac{\sup_{b\in B}\lvert b.F_{n} \triangle F_{n}\rvert}{\lvert F_{n}\rvert}\right)
	\label{eq:vdc-pre}
	\end{equation}
	where $\gamma_{f,x}(b_1 , b_2)$ stands for the differentiated term -
	\begin{equation*}
	\gamma_{f,x}(b_{1},b_{2}) = \frac{1}{\lvert F_n \rvert}\int_{g\in F_{n}}f(b_{1}.g.x)\overline{f(b_{2}.g.x)}dg.
	\end{equation*}
\end{lem}
\begin{proof}
    For a fixed element $b\in B$, one gets the following inequality
    \begin{equation}
    \begin{split}
        \left\lvert \frac{1}{\lvert F_{n}\rvert}\int_{g\in F_{n}}f(g.x)dg -\frac{1}{\lvert F_{n}\rvert}\int_{g\in F_{n}}f(b.g.x)dg  \right\rvert &\leq \frac{2}{\lvert F_{n} \rvert}\int_{g\in b.F_{n}\triangle F_{n}}\lvert f(g.x)\rvert dg \\
        &\leq 2\lVert f \rVert_{\infty}\frac{\lvert b.F_{n}\triangle F_{n}\rvert}{\lvert F_{n} \rvert}.
    \end{split}
    \end{equation}
    Now instead of considering a fixed element $b\in B$ one may average over all $b\in B$ in order to deduce
    \begin{equation*}
        \begin{split}
         \left\lvert \frac{1}{\lvert F_{n}\rvert}\int_{g\in F_{n}}f(g.x)dg -\frac{1}{\lvert F_{n}\rvert}\int_{g\in F_{n}}\frac{1}{\lvert B \rvert}\int_{b\in B}f(b.g.x)dbdg  \right\rvert &\leq \frac{2}{\lvert F_{n} \rvert}\frac{1}{\lvert B \rvert}\int_{b\in B}\int_{g\in b.F_{n}\triangle F_{n}}\lvert f(g.x)\rvert dgdb \\
        &\leq 2\lVert f \rVert_{\infty}\frac{1}{\lvert B \rvert}\int_{b\in B}\frac{\lvert b.F_{n}\triangle F_{n}\rvert}{\lvert F_{n} \rvert}db.   
        \end{split}
    \end{equation*}
    Bounding trivially the integral we have that
    \begin{equation}
    \begin{split}
         \left\lvert \frac{1}{\lvert F_{n}\rvert}\int_{g\in F_{n}}f(g.x)dg\right\rvert \leq \left\lvert \frac{1}{\lvert F_{n}\rvert}\int_{g\in F_{n}}\frac{1}{\lvert B \rvert}\int_{b\in B}f(b.g.x)dbdg  \right\rvert +  2\lVert f \rVert_{\infty}\frac{\sup_{b\in B}\lvert b.F_{n} \triangle F_{n}\rvert}{\lvert F_{n}\rvert}. 
        \end{split}    
    \end{equation}
	Using the Cauchy-Schwarz inequality for the expression $\frac{1}{\lvert F_{n}\rvert}\int_{g\in F_{n}}\frac{1}{\lvert B \rvert}\int_{b\in B}f(b.g.x)dbdg$ gives rise to the required inequality.
\end{proof}
Now we are in position to prove the main result of this section.
\begin{proof}[Proof of Theorem~\ref{thm:non-effective}]
By the approximation property, it is enough to prove Theorem~\ref{thm:non-effective} in the case where $f_{2}(g.y)=\psi(g.y)$ - a nilcharacter defined over $N/\Lambda$.
We fix some compact neighborhood $B^{H}_{1}$ of the identity $e\in H$, and we will take $\vol(B)\to\infty$ in the end, where $\vol$ stands for the Haar measure of $H$. We denote $F_{n}=B_{n}^{H}$ the sets $B_{n}^{H}=a_{\log n}B^{H}_{1}a_{-\log n}$ as the F\o lner sequence we use in the proof.
Moreover, by subtracting $\int_{X}f_{1}(x)d\mu(x)$, we may and will assume that $f_{1}$ is a function of vanishing integral. 

Using the Van-der-Corput type inequality~\eqref{eq:vdc-pre}, we may bound the twisted ergodic average in the following manner -
\begin{equation}
\left\lvert \frac{1}{\lvert F_n \rvert}\int_{g\in F_{n}}\psi(g.y)f_1(g.x)dg \right\rvert^{2} \ll \frac{1}{\lvert B \rvert^{2}}\iint_{b_1 , b_2\in B}\left\lvert \gamma_{f}(F_n,b_1,b_2) \right\rvert db_1 db_2+O_{f_1,\psi}\left(\frac{\sup_{b\in B}\lvert b.F_{n} \triangle F_{n}\rvert}{\lvert F_{n}\rvert} \right),
\label{eq:multidim-vdc}
\end{equation}
where $\gamma_{f}(F_n ,b_1,b_2)$ stands for the ''differentiated'' expression -
\begin{equation*}
\gamma_{f}(F_n ,b_1 , b_2)=\frac{1}{\lvert F_n \rvert}\int_{v\in F_n}\psi(b_1.v.y)\overline{\psi(b_2.v.y)}f_1(b_1.v.x)\overline{f_1(b_2.v.x)}dv.
\end{equation*}
By the differentiation property of nilcharacters, the ''differentiated'' nilcharacter $\psi(b.\cdot)\cdot\overline{\psi(\cdot)}$ is a nilsequence of degree less than $\psi$, and by the approximation property, can be uniformly approximated by a combination nilcharacters of degree strictly smaller than $\psi$, hence by induction on the degree of the nilsequence we have that
\begin{equation}\label{eq:base-case}
    \frac{1}{\lvert F_{n} \rvert}\int_{v\in F_{n}}\psi(b_1.v.y)\overline{\psi(b_2v.y)}\left(f_1(b_1.v.x)\cdot\overline{f_1(b_2.v.x)}-\rho_{f_1,f_1}(b_1 \cdot b_2 ^{-1})\right)dv \to 0,
\end{equation}
as $F_{n}\to H$, for any fixed $b_{1},b_{2}\in B$, where $\rho_{f_1,f_2}(g)$ stands for the matrix coefficient which is defined by $f_{1},f_{2}$ in the unitary representation of $G$ on $L_{0}^{2}(G/\Gamma)$, namely
\begin{equation*}
\rho_{f_{1},f_{2}}(g)=\int_{X}f_{1}(g.x)\overline{f_{2}(x)}d\mu(x).
\end{equation*}
As $B$ is compact, and the dependence of~\eqref{eq:base-case} in $b_{1},b_{2}$ is continuous, we have that
\begin{equation}
\frac{1}{\lvert B \rvert^{2} }\iint_{b_1 , B_2 \in B} \frac{1}{\lvert F_{n} \rvert}\int_{v\in F_{n}}\psi(b_1.v.y)\overline{\psi(b_2v.y)}\left(f_1(b_1.v.x)\cdot\overline{f_1(b_2.v.x)}-\rho_{f_1,f_1}(b_1 \cdot b_2 ^{-1})\right)dv \to 0,
\end{equation}
as $F_{n} \to H$, uniformly over $b_{1},b_{2}\in B$ due to compactness.
Hence we deduce the following bound:

\begin{equation}\label{eq:main-non-eff-estimate}
\begin{split}
&\left\lvert\frac{1}{\lvert F_n \rvert}\int_{v \in F_n}\psi(v.y)f_1(v.x)dv \right\rvert^{2} \\
&\ll \frac{1}{\lvert B \rvert^{2}}\iint_{b_1 , b_2\in B}\bigg\lvert\frac{1}{\lvert F_n \rvert}\int_{v\in F_n}\psi(b_1.v.y)\overline{\psi(b_2 .v.y)} \\
&{} \ \ \ \ \ \cdot\left(f_1(b_1.v.x)\cdot\overline{f_1(b_2.v.x)}-\rho_{f_1,f_1}(b_1\cdot b_2^{-1})\right)dv \bigg\rvert db_1 db_2 \\
&+ \frac{1}{\lvert B \rvert^{2}}\iint_{b_1 , b_2\in B}\left\lvert \frac{1}{\lvert F_n \rvert}\int_{v\in F_n}\psi(b_1.v.y)\overline{\psi(b_2.v.y)}\rho_{f_1,f_1}(b_1 \cdot b_2^{-1})dv\right\rvert db_1 db_2 \\
&+O_{f_1,\psi}\left(\frac{\sup_{b\in B}\lvert b.F_{n} \triangle F_{n}\rvert}{\lvert F_{n}\rvert}\right)^2 \\
&\ll  \frac{1}{\lvert B\rvert^{2}}\iint_{b_1 , b_2\in B}\bigg\lvert \frac{1}{\lvert F_n \rvert}\int_{v\in F_n} \psi(b_1.v.y)\overline{\psi(b_2.v.y)}\\ 
&{} \ \ \ \ \ \cdot\left(f_1(b_1.v.x)\cdot\overline{f_1(b_2.v.x)}-\rho_{f_1,f_1}(b_1\cdot b_2 ^{-1})\right)dv \bigg\rvert db_1 db_2 \\
&+  \frac{\|\psi\|_{\infty}^{2}}{\lvert B \rvert^{2}}\iint_{b_1,b_{2}\in B}\left\lvert\rho_{f_1,f_1}(b_{1}\cdot b_{2}^{-1})\right\rvert db_1 db_2\\
&+O_{f_1,\psi}\left(\frac{\sup_{b\in B}\lvert b.F_{n} \triangle F_{n}\rvert}{\lvert F_{n}\rvert}\right)^2.
\end{split}
\end{equation}
The first summand converges to $0$, as $n\to \infty$ by induction hypothesis.
For the second summand, by the Howe-Moore theorem~\cite[Theorems~$5.2$,$6.1$]{howemoore79}, we have that $\rho_{f,f}(b) \to 0$ as $b\to\infty$ in $G$, therefore the Cesaro average $\frac{1}{\lvert B \rvert^{2}}\iint_{b_1, b_{2}\in B}\left\lvert\rho_{f,f}(b_1 \cdot b_{2}^{-1})\right\rvert$ tends to $0$ as well, as this follows from say projecting the Cesaro average to the abelianization of $H$. The third summand converges to $0$ as $n\to\infty$ by the assumption that $\left\{F_{n}\right\}$ is a F\o lner sequence.
\end{proof}

\section{Proof of the quantitative horospherical equidistribution theorem}\label{sec:effective}
In order to deduce an effective version of Theorem~\ref{thm:non-effective}, we need to control the three summands in~\eqref{eq:main-non-eff-estimate}.
We start by presenting the following estimate regarding the F\o lner sequence decay.
\begin{prop}
For $B^{H}_{R}=a_{\log R}B^{H}_{1}a_{-\log R}$, we have the following estimate for and $y\in B^{H}_{b}$ for $b<R$:
\begin{equation*}
    \frac{\lvert b.B^{H}_{R}\triangle B^{H}_{R}\rvert}{\lvert B^{H}_{R}\rvert} \ll \frac{\lVert \log(b) \rVert}{R},
\end{equation*}
\end{prop}
where $\lVert \cdot \rVert$ is a fixed norm on $Lie(G)$, which we choose for convenience to be the infinity norm with respect to a fixed basis.
\begin{proof}
Let $b\in B_{b}^{H},x\in B_{R}^{H}$ be given, we may write $b=\exp(\underline{b}),$ $x=~a_{\log R}\exp(\underline{x})a_{-\log R}$ for some $\underline{b}\in \log(B_{b}^{H}), \underline{x}\in \log(B_{1}^{H})$.
We note here that we have the following equality
\begin{equation*}
\begin{split}
    b.x &= \exp(\underline{b}).a_{\log R}\exp(\underline{x})a_{-\log R} \\
    &= a_{\log R}\exp(\underline{b}/R)\exp(\underline{x})a_{-\log R}.
\end{split}
\end{equation*}
Using the Baker-Campbell-Hausdorff formula (noticing it terminates after finitely-many terms, as all calculations are done inside the nilpotent group $H$), we see that 
$$ \log(\exp(\underline{b}/R)\exp(\underline{x})) = \underline{b}/R+\underline{x}+ \text{higher order terms}, $$
where the (finitely many) higher order terms are composed of commutators of $\underline{x},\underline{b}/R$.
It is evident that the commutators are of norm of at-most $O(\lVert \underline{b}/r\rVert)$ for an absolute constant which depends on the group, hence 
$$ b.B^{H}_{R}\triangle B^{H}_{R} \subset  a_{\log R}\left( B^{H}_{1+O(\lVert \underline{b}/r\rVert)}\setminus B^{H}_{1-O(\lVert \underline{b}/r\rVert)}\right)a_{-\log R}. $$
Using the volume formula for the balls $\left\{B^{H}_{R}\right\}$ in~\eqref{eq:ball-vol}, we see that this volume is bounded by $2^{d_{H}}\cdot 2\cdot O(\lVert b \rVert) \cdot R^{d_{H}-1}\cdot \vol_{H}(B^{H}_{1})$.
Dividing by $\vol_{H}(B^{H}_{R})$ we get the desired estimate.
\end{proof}

We remark here that in the case of changing the F\o lner sequence to a more general sequence of sets, such as metric balls, one can still recover such a boundary type estimate (with worse exponent) by using the results of R. Tessera~\cite[Theorem~$4$]{Tessera2007}.

The second summand is controlled by quantitative estimates about the decay of matrix coefficient for semisimple Lie groups:
\begin{thm}\label{thm:HC-mixing}[Harish-Chandra bound,~\cite{HC58} Theorem~$3$, \cite{GV88} Theorem~$4.6.4$, \cite{Haagerup1988} Theorem~$2$, \cite{Katok1994} Theorem~$3.1$, \cite{oh2002} Theorem~$1.1$] Let $G$ be a semisimple Lie group without compact factors, given $\Gamma\leq G$ a lattice, there exists some $s'=s'(\Gamma)>0$ for which
\begin{equation}\label{eq:HC-bound}
\left\lvert \int_{G/\Gamma}g.f_{1}(x)\cdot \overline{f_{2}(x)}d\mu(x) \right\vert \ll \|\underline{g}\|^{-s'}Sob(f_{1})Sob(f_{2})
\end{equation}
for $g=\exp(\underline{g})$, for some Sobolev norms of $f_1,f_{2}$, for any two smooth and bounded functions $f_{1},f_{2}$ with vanishing integrals and $\lVert g \rVert$ is a fixed norm on $Lie(G)$.
\end{thm}
Using Harish-Chandra's multiplicity bound, one may take Sobolev norms of $2(\dim K +1)$ order, where $K\leq G$ is the maximal compact subgroup of $G$.
A direct corollary of this bound for sets of the form $B=B^{H}_{R}$ is that $\frac{1}{\lvert B \rvert}\iint_{ b \in B^{-1}\cdot B}\rho_{f,f}(b)$ is bounded by $\lvert B \rvert^{-s'}$ for some explicitly computable $s'=s'(\Gamma)$ up to constants which depend on the regularity of $f$.

Handling the first summand in~\eqref{eq:main-non-eff-estimate}, requires a derivation of an effective version of a horospherical equidistribution theorem by Dani~\cite[Theorem $A$]{dani1986}, in order to make the base case of the induction effective.
We start by proving a quantitative version of the pointwise ergodic theorem, afterwards we will prove a quantitative version of the disjointness theorem by means of induction over the degree of the nilcharacater.
In this section, we derive a quantitative version for the pointwise ergodic theorem, where the induction will be carried in the next section.


We start by introducing the following lemma which serves as an approximate mean ergodic theorem.
\begin{lem}\label{lem:effective-ergodic}
Given a smooth compactly-supported function $f$ with vanishing integral, the following measure estimate holds:
\begin{equation*}
\mu\left(\left\{x\in G/\Gamma : \left\lvert\frac{1}{\vol_{H}\left(B^{H}_{R}\right)}\int_{u\in B^{H}_{R}}f(u.x)du \right\rvert \geq \vol_{H}\left(B^{H}_{R}\right)^{-\delta} \right\} \right) \ll_{f} \vol_{H}\left(B^{H}_{R}\right)^{2\delta-2s'},
\end{equation*}
for some explicit $s'=s'(\Gamma)$.
\end{lem}
We note here that the $f$-dependence is the same one as given in the mixing rate bound~\eqref{eq:HC-bound}.
\begin{proof}
By Chebyshev's inequality we obtain that
\begin{align*}
&\mu\left(\left\{x\in G/\Gamma :\left\lvert\frac{1}{\vol_{H}\left(B^{H}_{R}\right)}\int_{u \in B^{H}_{R}}f(u.x)du \right\rvert \geq \vol_{H}\left(B^{H}_{R}\right)^{-\delta} \right\} \right) &\\
&\leq \vol_{H}\left(B^{H}_{R}\right)^{2\delta}\left\|\frac{1}{\vol_{H}\left(B^{H}_{R}\right)}\int_{u \in B^{H}_{R}}f(u.x)du \right\|_{L^2(G/\Gamma)}^{2}.
\end{align*}
Expanding $\left\|\frac{1}{\vol_{H}\left(B^{H}_{R}\right)}\int_{u\in B^{H}_{R}}f(u.x)du \right\|_{L^2(G/\Gamma)}^{2}$ according to the $L^{2}$ inner-product we have the following -
\begin{equation*}
\begin{split}
&\left\|\frac{1}{\vol_{H}\left(B^{H}_{R}\right)}\int_{u\in B^{H}_{R}}f(u.x)du \right\|_{L^2(G/\Gamma)}^{2} = \int_{w\in \left(B^{H}_{R}\right)^{-1}\cdot\left(B^{H}_{R}\right)\setminus\left\{e\right\}}\lvert\rho_{f,f}(w)\rvert d\nu^{H}_{R}(w),
\end{split}
\end{equation*}
where $\nu^{H}_{R}$ stands for the (left-)convolution measure on $H$ defined by
\begin{equation*}
\nu^{H}_{R} = \vol_{H}\mid_{B^{H}_{R}} * \vol_{H}\mid_{B^{H}_{R}},
\end{equation*}
and $\vol_{H}\mid_{B^{H}_{R}}$ is the normalized measure achieved by restricting and normalizing the Haar measure defined on $H$ to $B^{H}_{R}$.
Similarly to before, by convexity and the Harish-Chandra bound~\eqref{eq:HC-bound} for the matrix coefficients, the last term is bounded by $O_{f}\left(\vol_{H}\left(B^{H}_{R}\right)^{-s'}\right)$.
\end{proof}
The next lemma quantifies the polynomial divergence of nearby points, namely it gives quantitative control over the deviation of the ergodic averages sampled along the trajectories of two nearby origin points.
In the proceeding lemma, $R$ should be thought of a fixed large constant, and $\delta<1$ is a small number.
\begin{lem}[Key Lemma] Let $A=\left\langle a_{t}\right\rangle\leq G$ be a one-parameter subgroup of $G$ such that $H$ is horospherical with respect to $A$. Assume that $dist(a_{-\log{R}}.x,a_{-\log{R}}.y)<\delta$ for some $x,y\in G/\Gamma$, for any $\delta<1$, then
\begin{equation}
\left\lvert \frac{1}{\vol_{H}\left(B^{H}_{R}\right)}\int_{u\in B^{H}_{R}}f(u.x)-f(u.y)du \right\rvert \leq_{f} \delta.
\label{eq:key-lemma}
\end{equation}
\label{lem:key-lemma}\end{lem}

\begin{proof}
	By a change of variable and the renormalization properties of the $A$-action with respect the horospherical subgroup we have that -
	\begin{equation}
\label{eq:renormalization}
	\begin{split}
	&\left\lvert\frac{1}{\vol_{H}\left(B^{H}_{R}\right)} \int_{u\in B^{H}_{R}}f(u.x)du-\frac{1}{\vol_{H}\left(B^{H}_{R}\right)} \int_{u \in B^{H}_{R}}f(u.y)du \right\rvert \\
	&= \left\lvert\int_{u \in B^{H}_{1}}f(\left(a_{\log R}ua_{-\log R}\right).x)du-\int_{u \in B^{H}_{1}}f(\left(a_{\log R}ua_{-\log R}\right).y)du\right\rvert \\
	&= \left\lvert\int_{u \in B^{H}_{1}}f(a_{\log R}.u.x')du-\int_{u \in B^{H}_{1}}f(a_{\log R}u.y')du\right\rvert,
	\end{split}
	\end{equation}
	where we write $x'=a_{-\log R}.x,\ y'=a_{-\log R}.y$.
	Writing $y'=\varepsilon.x'$ for some matrix $\varepsilon$ for which $\left\|\varepsilon-Id\right\|<\delta$, we see that
	\begin{equation*}
	u.\varepsilon.x'=(u\varepsilon u^{-1})u.x'.
	\end{equation*}
	We have the following splitting of $Lie(G)$ according to the Adjoint action of $\left<a\right>$ -
	\begin{equation*}
	Lie(G)=\mathfrak{g}^{-}\oplus\mathfrak{g}^{0}\oplus\mathfrak{g}^{+},
	\end{equation*}
	where $\mathfrak{g}^{-},\mathfrak{g}^{0},\mathfrak{g}^{+}$ stands for the sum of the negative eigenspaces, zero eigenspaces and positive eigenspaces accordingly.
On an open and dense subset of $G$ containing the identity, the map $\mathfrak{g}^{-}\oplus \mathfrak{g}^{0} \oplus\mathfrak{g}^{+}\to G$ given by the composition of exponentiation and group multiplication is bi-regular (c.f. \cite[Proposition~$2.7$]{margulisTomanov}).
We may assume that $\varepsilon$ belongs to that dense open set and we denote $\underline{\varepsilon}\in Lie(G)$ to be the element such that $\exp(\underline{\varepsilon})=\varepsilon$.
Furthermore, we may express $\underline{\varepsilon}$ according the the Lie algebra splitting as
$$\underline{\varepsilon} = \underline{v}^{-} + \underline{v}^{0}+\underline{v}^{+}, $$
where $\underline{v}^{-}\in \mathfrak{g}^{-}, \underline{v}^{0}\in \mathfrak{g}^{0}, \underline{v}^{+}\in \mathfrak{g}^{+}$, and $\lVert \underline{v}^{-} \rVert, \lVert \underline{v}^{0} \rVert, \lVert \underline{v}^{-}\rVert \ll~\delta$.
For given $\underline{t}\in B_{1}^{H}$, we define $\underline{s}(\underline{t}) \in \mathfrak{g}^{+}$ by the following relation:
$$ \log\left(\exp(\underline{t})\cdot \exp( \underline{v}^{-} + \underline{v}^{0}) \cdot \exp(-\underline{t})\cdot
\exp(-\underline{s})\right) \in \mathfrak{g}^{-}+\mathfrak{g}^{0}. $$
Clearly we have that $\underline{s}(\underline{0})=\underline{0}$. This relation defines $\underline{s}$ as a ratio of \emph{polynomial} functions in $\underline{t}$ as an application of the implicit function theorem, applied to the system of equations which one gets after applying the inverse of the bi-regular map and demanding that the resulting elements in the Lie algebra will be orthogonal to $\mathfrak{g}^{+}$ (with respect to the Killing form).
Moreover we have that $$\left\lvert \frac{\partial}{\partial \underline{t_{i}}}\underline{s}(\underline{t}) \right\rvert \ll O(\delta)$$ for every $\underline{t_{i}}$ in some fixed basis of $\mathfrak{g}^{+}$.
Therefore $\lVert\underline{t}+\underline{s}(\underline{t}) \rVert = O(1)$ for $\lVert\underline{t}\rVert \leq O(1)$, hence $\lVert \exp(\underline{t})\cdot \exp( \underline{v}^{-} + \underline{v}^{0}) \cdot \exp(-\underline{t})\cdot
\exp(-\underline{s}) \rVert = O(\delta)$.
Denoting $v(u)=\exp(\underline{s}(\underline{t}))$ for $u=\exp(\underline{t})$ we have
\begin{equation*}
\begin{split}
    \int_{u \in B^{H}_{1}}f(a_{\log R}u.y')du &= \int_{u \in B^{H}_{1}}f(a_{\log R}u.\varepsilon x')du \\
    &= \int_{u \in B^{H}_{1}}f(a_{\log R}(u\varepsilon^{-}\varepsilon^{0}u^{-1}v^{-1}).vu\varepsilon^{+} x')du \\
    &= \int_{u \in B^{H}_{1}}f(a_{\log R}vu\varepsilon^{+} x')du+O_{f}(\delta),
\end{split}
\end{equation*}
as the $A$-action is non-expanding along $\mathfrak{g}^{-}\oplus\mathfrak{g}^{0}$.
Writing $\log(vu\varepsilon^{+})\in\mathfrak{g}^{+}$ according to the Baker-Campbell-Hausdorff formula, as the exponential mapping is onto in nilpotent groups, we see that
$$\log(vu\varepsilon^{+}) = \underline{t}+\underline{s}(\underline{t})+\underline{\varepsilon}^{+}+\underline{w}(\underline{t})\in\mathfrak{g}^{+}, $$
where $\underline{w}$ is composed of finitely many commutators involving at-least one of either $\underline{s}$ or $\underline{\varepsilon}^{+}$ and hence $\lVert\underline{w}(\underline{t})\rVert=O(\delta)$.
Set the following change of variables: $$\underline{r}(\underline{t})=\log(vu\varepsilon^{+})\in\mathfrak{g}^{+}.$$ Then the following holds:
\begin{equation*}
    D_{\underline{t}}\underline{r} = I+D_{\underline{t}}\underline{s}+D_{\underline{t}}\underline{w}.
\end{equation*}
We have that
\begin{equation*}
    J(\underline{r}) = \det D_{\underline{t}}(\underline{r})=\det\left(I+q(\underline{t})\right),
\end{equation*}
for some continuous function $q$ which is bounded by $O(\delta)$ for $\lVert t \rVert \leq 1$, as $$\lVert D_{\underline{t}}\underline{s} \rVert, \lVert D_{\underline{t}}\underline{w}\rVert \ll O(\delta).$$
Using Taylor expansion we get 
\begin{equation*}
	\det(I+q(\underline{t}))=1+\tr(q(\underline{t}))+O(\delta^2),
\end{equation*}
where also $\tr(q(\underline{t}))=O(\delta)$ for $\lVert t \rVert \ll O(1)$.
Hence we have that
\begin{equation}
    \int_{u \in B^{H}_{1}}f(a_{\log R}vu\varepsilon^{+} x')du = \int_{u \in B^{H}_{1}}f(a_{\log R}u x')du + O_{f}(\delta),
\end{equation}
this concludes the proof of the lemma.

\end{proof}
We note here that the $f$-dependence in the above theorem (for our choice of the class of functions) is only in $\lVert f \rVert_{\infty}$.

Now we may complete the proof of the equidistribution theorem, for the case where $\Gamma$ is a uniform lattice.

\begin{proof}[Proof of Theorem~\ref{thm:effective-equi} where $\Gamma$ is a uniform lattice]
Let $x\in G/\Gamma$ be a given point in the homogenous space. Define $x'$ to be $x'=a_{-\log{R}}.x$.
Let $\delta \ll InjRad(X)$ be a small parameter to be chosen later, where $InjRad(X)$ stands for the minimal injectivity radius of points in $X$ which is positive by the assumption that $\Gamma$ is uniform lattice.
Look at $B_{\delta}(x') \subset X$. By Lemma~\ref{lem:effective-ergodic} there exists some $y' ~\in B_{\delta}(x')$ for which $\left\lvert\frac{1}{\vol_{H}\left(B^{H}_{R}\right)}\int_{u \in B^{H}_{R}}f(u.y)du \right\rvert \leq \vol_{H}\left(B^{H}_{R}\right)^{-\gamma}$ for some $\gamma~<~s$ and $y~=~a_{\log R}.y'$, as long as $\vol(B_{\delta}(x'))\geq \vol_{H}\left(B^{H}_{R}\right)^{2\gamma-2s}$. As $\vol(B_{\delta}(x'))\approx~\delta^{\dim G}$ we require that
\begin{equation}
\delta \geq \vol_{H}\left(B^{H}_{R}\right)^{\frac{2\gamma-2s}{\dim G}},
\label{eq:delta-condition-acc}
\end{equation}
and in particular it would be sufficient to require
\begin{equation}
\delta \gtrsim R^{\left(2\gamma-2s\right)\frac{d_{H}}{\dim G}}.
\label{eq:delta-condition}
\end{equation}

By Lemma~\ref{lem:key-lemma} we have that
\begin{equation*}
\left\lvert\frac{1}{\vol_{H}\left(B^{H}_{R}\right)}\int_{u\in B^{H}_{R}}f(u.y)-f(u.x)du \right\rvert \leq_{f} \delta.
\end{equation*}
Hence by the triangle inequality we deduce the following bound:
\begin{equation}
\left\lvert\frac{1}{\vol_{H}\left(B^{H}_{R}\right)}\int_{u\in B^{H}_{R}}f(u.x)du \right\rvert \leq_{f} \vol_{H}\left(B^{H}_{R}\right)^{-\gamma}+\delta,
\label{eq:main-estimate}
\end{equation}
where the $f$ dependence is given by the Lipschitz norm of $f$.

Optimizing for $\delta$ under condition~\eqref{eq:delta-condition} we may take $\gamma = \frac{2s}{2+\dim G}$, which yields a bound of $$O_{f}\left(\vol_{H}\left(B^{H}_{R}\right)^{-\frac{2s}{2+\dim G}}\right)=O_{f,H}\left(R^{-\frac{2s\cdot d_{H}}{2+\dim G}}\right)$$ in~\eqref{eq:main-estimate}.
\end{proof}
\begin{rem}
	A slightly more careful optimization procedure would have yield \begin{equation*}
	\gamma=\frac{2s}{\dim G_{a}^{+}+\dim G_{a}^{0} +2},
	\end{equation*}
	where $G_{a}^{+}$ is the stable horospherical group defined by $a$, and $G_{a}^{0}$ is the natural subgroup defined by $\exp\left(\mathfrak{g}_{0}\right)$.
\end{rem}
A key ingredient in the proof of the non-uniform case of Theorem~\ref{thm:effective-equi} in our method is the notion of $\Theta$-diophantine points.
Before we present the details, we show how being $\Theta$-diophantine point (as in Definition~\ref{def:dioph}) allows us to bootstrap quantitative recurrence estimates to control orbit cusp excursions.
Those quantitative estimates were first proven by Kleinbock-Margulis~\cite{KleinbockMargulis}, strengthening an earlier non-divergence theorem of Dani-Margulis. Below we give a brief introduction to the definitions and estimates used in their proof, for a more thorough introduction the reader may consult~\cite{kleinbock-pisa}. We assume throughout the rest of the section that $R\geq 1$.
\begin{defn}\label{def:C-alpha-good}
A function $f:X\to \mathbb{R}$ where $X$ is a locally compact metric space is called \emph{$(C,\alpha)$-good function} with respect to a measure $\mu$ for some $C,\alpha>0$ if it satisfies the following estimate for every open convex subset $B\subset X$:
\begin{equation}\label{eq:c-a-good-def}
\forall \varepsilon>0 \ \ \mu\left(\left\{x\in B \mid \lvert f(x) \rvert <\varepsilon \right\}\right) \leq C\cdot \left(\frac{\varepsilon}{\sup_{x\in B} \lvert f(x)\rvert} \right)^{\alpha} \mu(B).
\end{equation}
\end{defn}
The primary source of examples for $(C,\alpha)$-good functions are polynomial functions (c.f. \cite[Proposition~$3.2$]{KleinbockMargulis}), as we have the following inequality due to Remez (c.f. \cite[Equation~$(14)$]{BrudnyiGanzburg}), for any polynomial mapping $f:\mathbb{R}^{n} \to \mathbb{R}$ of degree at most $d$ defined on a convex subset $B$ of $\mathbb{R}^{n}$:
\begin{equation}\label{eq:remez-ineq}
\left\lvert \left\{t\in B \mid \lvert f(t) \rvert \leq \varepsilon \right\} \right\rvert \leq 4\cdot n \cdot \left(\frac{\varepsilon}{\sup_{x\in B}\lvert f\rvert} \right)^{1/d}\lvert B \rvert.
\end{equation}
As unipotent actions over homogeneous spaces are of polynomial nature, studying trajectories of unipotent flows leads naturally to the class of $(C,\alpha)$-good functions, by identifying the flow on the homogeneous space with the associated flow in the Lie algebra.

We think of our algebraic group as coming with $\mathbb{Q}$-rational structure, namely we may assume that $G$ is embedded in $SL_{k}(\mathbb{R})$ for some $k$, and the homogeneous space $G/\Gamma$ is properly embedded into $SL_{k}(\mathbb{R})/SL_{k}(\mathbb{Z})$.
Given $v\in\mathbb{Z}^{k}\setminus\{0\}$ we define the following functions \\ $\psi_{R,v,x_0}(\underline{h}):~\log B^{H}_{R}\subset~Lie(H)~\to~\mathbb{R}$ by
\begin{equation}
\psi_{R,v,x_0}(\underline{h}) = \left\| \left(\exp(\underline{h}).x_0\right).v \right\|.
\label{eq:norm-of-vec}
\end{equation}
We further define for any primitive subgroup $\Delta \leq \mathbb{Z}^{k}$ the functions $\psi_{R,\Delta,x_0}(\underline{h}):~\log B^{H}_{R}\subset Lie(H) \to \mathbb{R}$ by
\begin{equation*}
\psi_{R,\Delta,x_0}(\underline{h}) = \left\| \left(\exp(\underline{h}).x_0\right).\Delta \right\|.
\label{eq:norm-of-subgrp}
\end{equation*}
\begin{defn}
Let $(X,d)$ be a metric space equipped with a transitive $H$-action. A Borel measure $\mu$ on $X$ is called \emph{uniformly Federer} with respect to a family of subsets $\left\{B_{R}\right\}$ if there exists a constant $C>0$ such that for every $R>0$ and every $h\in H$ the following holds:
\begin{equation}
\frac{\mu(h.B_{3R})}{\mu(h.B_{R})} \leq C.
\label{eq:Federer-def}
\end{equation}
\end{defn}
\begin{lem}
The natural measure $\mu$ on $Lie(H)$ is \emph{uniformly Federer} with respect to the family $\left\{Ad_{a_{\log R}}\log B^{H}_{1} \right\}$. More explicitly:
\begin{equation}
\frac{\mu(Ad_{a_{\log 3R}}\log B^{H}_{1})}{\mu(Ad_{a_{\log R}}\log B^{H}_{1})} \leq C.
\label{eq:Federer}
\end{equation}
\end{lem}
The proof follows at once as $\mu$ satisfies a power law.

\begin{lem}[Nilpotent averages are $(C,\alpha)$-good]\label{lem:nil-c-alpha}
Given $R,\Delta,x_0$, where $\Delta$ is a primitive subgroup $\Delta\leq\mathbb{Z}^{k}$, the function $\psi_{R,\Delta,x_0}$ is a $(C,\alpha)$-good function for some $C,\alpha>0$.
\end{lem}
\begin{proof}
Explicitly writing the function in the exterior product $\largewedge^{k}(\mathbb{R})$, we see this function is polynomial and by applying~\eqref{eq:remez-ineq}, and noticing that $\mu$ is uniformly Federer and absolutely continuous with respect to the Lebesgue measure we deduce the result.
Namely we have the following estimate:
\begin{equation}
\begin{split}
\mu( \left\{\underline{h}\in \log B^{H}_{R} \mid \psi_{R,\Delta,x_0}(\underline{h}) \leq \varepsilon \right\}) &\leq C\cdot \left(\frac{\varepsilon}{\sup_{\underline{h}\in \log B^{H}_{R}}\psi_{R,\Delta,x_0}(\underline{h})} \right)^{\alpha} \mu\left( \log B^{H}_{R} \right) \\
&\leq C\cdot \left(\frac{\varepsilon}{\psi_{R,\Delta,x_0}(\underline{0})} \right)^{\alpha}\mu\left( \log B^{H}_{R} \right),
\end{split}
\label{eq:C-a-good-horo}
\end{equation}
for some explicit $\alpha$.
\end{proof}

Recall that a Besicovitch space is a metric space where the Besicovitch covering theorem (c.f.~\cite[Theorem~$2.7$]{mattila_1995} holds.
The quantitative non-divergence estimate of Kleinbock-Margulis as generalized by Kleinbock-Lindenstrauss-Weiss and Kleinbock takes the following form:
\begin{thm}[\protect{\cite[Theorem~$5.2$]{KleinbockMargulis},\cite[Theorem~$4.3$]{Kleinbocklindenstraussweiss},\cite[Theorem~$2.2$]{Kleinbock}}]\label{thm:KM}
Given an open set $U\subset X$ of some Besicovitch metric space $X$, positive constants $C,D,\alpha$ and a measure $\mu$ which is uniformly-Federer on $U$, there exists $C'>0$ with the following property. Suppose $h:~U~\to~SL_{k}(\mathbb{R})$ is a continuous map, $0<\rho \leq 1$, $z\in U\cap \supp \mu$ and $B=B(z,r)$ is a ball such that $\tilde{B}=B(z,3^{k-1}r)$ is contained in $U$, and for every primitive subgroup $\Delta \leq \mathbb{Z}^{k}$ the following holds, for the set functions $\phi_{\Delta}(x):=\lVert h(x).\Delta \rVert$:
\begin{itemize}
	\item The function $\phi_{\Delta}(x)=\lVert h(x).\Delta \rVert$ is a $(C,\alpha)$-good function on $\tilde{B}$ with respect to $\mu$.
	\item $\sup_{x\in B\cap supp \mu}\phi_{\Delta}(x) \geq \rho^{rank(\Delta)}$.
\end{itemize}
Then for every $0<\varepsilon\leq \rho$,
\begin{equation*}
\mu\left( \left\{x\in B \mid \pi(h(x))\notin X^{1}_{\geq\varepsilon} \right\} \right) \leq C'\cdot \left(\frac{\varepsilon}{\rho}\right)^{\alpha}\cdot \mu(B),
\end{equation*}
where $\pi$ denotes the projection $SL_{k}(\mathbb{R})\to SL_{k}(\mathbb{R})/SL_{k}(\mathbb{Z})$.
\end{thm}

We remark here the original formulation of Theorem~\ref{thm:KM} in~\cite{Kleinbocklindenstraussweiss} is slightly weaker, requiring that
$$ \sup_{x\in B\cap \supp\mu}\phi(x) \geq \rho, $$
for all primitive subgroups $\Delta \leq \mathbb{Z}^{k}$.
The generalization we present here appears in a later work by Kleinbock~\cite{Kleinbock}.

We note here that in general, the nilpotent group $H$ is not a Besicovitch space, hence we need to give a variant of the theorem~\cite{Kleinbocklindenstraussweiss} which is applicable to our case.
Moreover, by Lemma~\ref{lem:nil-c-alpha}, the functions $\phi(x)$ are $(C,\alpha)$-good.

\begin{defn}
A \emph{modified ball} of radius $r$ around $z$ in $Lie(H)$ is a subset of the following form
\begin{equation*}
B(z,r)=Ad_{a_{\log r}}\log B^{H}_{1}+z
\end{equation*}
\end{defn}

\begin{lem}[Modified Vitali covering lemma]\label{lem:vitali}
Let $\mathcal{F}$ be a collection of modified balls in $Lie(H)$ with bounded radii. There exists a disjoint subcollection $\mathcal{G}$ of modified balls drawn from $\mathcal{F}$ such that every modified ball $B'\in \mathcal{F}$ intersects some modified ball $B=B(z,r)\in\mathcal{G}$ and $B'\subset~B(z,C_{H}\cdot r)$ for some $C_{H}>0$.
\end{lem}
The proof follows along the lines of the usual proof of the Vitali covering lemma~\cite[Theorem~$2.8$]{mattila_1995}.

\begin{thm}[Modified non-divergence]\label{thm:mod-non-divergence}
Given an open set $U\subset~Lie(H)$, positive constants $C,D,\alpha$ and a measure $\mu$ which is uniformly Federer on $Lie(H)$ as above, then there exists $C'>0$ with the following property. Suppose $\exp:Lie(H)\to SL_{k}(\mathbb{R})$ is the exponential map, $0<~\rho~\leq~1$, $z\in Lie(H)\cap \supp \mu$ and $B=B(z,r)$ is the following subset -  $B(z,r)~=~Ad_{a_{\log r}}\log B^{H}_{1}~+~z$ such that $\tilde{B}=B(z,3\cdot C_{H}^{\dim H}\cdot r)$ is contained in $U$, and for every primitive subgroup $\Delta \leq \mathbb{Z}^{k}$ and its associated function $\phi_{\Delta}(\underline{h}):=\lVert \exp(\underline{h}).\Delta \rVert$ the following holds:
\begin{itemize}
	\item $\sup_{\underline{h}\in B(z,r)\cap supp \mu}\phi_{\Delta}(\underline{h}) \geq \rho^{rank(\Delta)}$.
\end{itemize}
Then for every $0<\varepsilon\leq \rho$,
\begin{equation*}
\mu\left( \left\{\underline{h}\in B \mid \pi(\exp(\underline{h}))\notin X^{1}_{\geq\varepsilon} \right\} \right) \leq C'\cdot \left(\frac{\varepsilon}{\rho}\right)^{\alpha}\cdot \mu(B),
\end{equation*}
where $\pi$ denotes the natural projection $SL_{k}(\mathbb{R})\to SL_{k}(\mathbb{R})/SL_{k}(\mathbb{Z})$.
\end{thm}

The proof follows at once from the proof given in~\cite[Theorems~$2.2$,$2.1$]{Kleinbock}.
The only modification needed in the proof is to replace the Besicovitch covering theorem used in equation~$(2.2)$ of~\cite{Kleinbock} by the Vitali covering theorem in the version stated in Lemma~\ref{lem:vitali}.
As a result, the union of modified balls of radius $C_{H}\cdot r_{y}$ would need to be taken, which is contained inside the modified ball $B(x,3\cdot C_{H}\cdot r)$, and we may estimate the measure of $B(x,3\cdot C_{H}^{\dim H} \cdot r)$ using the fact that $\mu$ is uniformly Federer.
The following corollary allows us to relate the diophantine type of the origin point $x$ with the excursion of other points in the piece of orbit $B^{H}_{R}.x$ in the following quantitative manner:

\begin{cor}\label{cor:inheritaed-diophantine-cond1}
Given some one-parameter digonalizable subgroup $A\leq~G$, assume that $x$ is a diophantine point of type $\Theta$ for $D=D(\Gamma)$ and $H\leq G$ which is horopsherical with respect to $A$, and let $B^{H}_{R}$ defined as before, then for any $R\geq \Theta(r):$
\begin{equation}\label{eq:bound-dioph}
\frac{\vol_{H}\left\{h \in B^{H}_{R} \mid ha_{-\log r}.x \notin X^{1}_{\geq r^{-D-\epsilon}}\right\}}{\vol_{H}\left(B^{H}_{R}\right)} \ll r^{-\alpha\cdot\epsilon},
\end{equation}
where $\alpha>0$ is the same constant as in Theorem~\ref{thm:mod-non-divergence}.
\end{cor}

The proof follows at-once from Theorem~\ref{thm:mod-non-divergence} as we may take $$\rho = C\cdot r^{-D}$$ by the diophantine condition imposed on the point $x$.

As was noted in the begining of the paper, whenever a point $p$ belongs to the set $X^{1}_{\geq r^{-D}}$, then we have a bound over the injectivity radius at $p$, $InjRad(p) \geq r^{-D\cdot k}$.
As a result, we may state the previous inequality as
\begin{equation*}
\frac{\vol_{H}\left\{h \in B^{H}_{R} \mid InjRad\left(a_{-\log r}h.x\right) \leq r^{-(D+\epsilon)\cdot k}\right\}}{\vol_{H}\left(B^{H}_{R}\right)} \ll C\cdot r^{-\alpha\cdot\epsilon},
\end{equation*}
for every $R>\Theta(r)$, whenever $x$ is diophantine point of type $\Theta$.

We note also that in the case where $x$ belongs to some absolute compact set, a direct corollary of Theorem~\ref{thm:mod-non-divergence} would be that most of the orbit piece $B^{H}_{R}.x$ belongs to the set $X^{1}_{\geq R^{-\epsilon}}$, by setting up $\rho = O_{x}(1)$.

We end the discussion of diophantine conditions by explicitly providing examples of such lattices, and in particular relating those conditions to algebraic properties of the lattice $x_0.\mathbb{Z}^{k}$ by means of Schmidt's subspace theorem using results of Skriganov~\cite{Skriganov}.

The examples we provide are based on slow divergence of the base point, namely $a_{-\log R}.x_0$ diverges slowly (with respect to $R^{-D}$), and hence $\Theta(R)=1$ for those examples.

We restrict ourselves to the case of $G=SL_{k}(\mathbb{R})$, $\Gamma=SL_{k}(\mathbb{Z})$.
By \cite[Lemma~$3.2$, Equation~$3.21$]{Skriganov} we have that for (Haar) almost-every lattice $x\in G/\Gamma$ and arbitrarily small $\epsilon>0$
\begin{equation*}
\min_{v\in a_{-r}.x\setminus\{0\}}\lVert v \rVert >C(\epsilon,x)\cdot r^{-1+\frac{1}{k}-\epsilon},
\end{equation*}
which in particular shows that our diophantine condition holds \emph{generically} with any $D>0$, as the definition of the semi-norm $\alpha$ involves computing such a minima over the countably-many primitive subgroups of $\mathbb{Z}^{k}$ and countable intersection of sets of full measure is of full measure.

For every $y\in \mathbb{R}^{k}$ we define the quantity $Nm(y) = \prod_{i=1}^{k}\lvert y_{i} \rvert$.
It is evident that $\lVert y \rVert_{\infty}^{k} \geq Nm(y)$.
Moreover, if $a\in SL_{n}(\mathbb{R})$ is a diagonal matrix  we have that $Nm(a.y)=Nm(y)$, and $\lVert a.y \rVert_{\infty} \geq a_{\min}\cdot \lVert y \rVert_{\infty}$, where $a_{\min}$ is the smallest entry (in absolute value) in the diagonal of the matrix $a$.

Now assume that $g\in SL_{k}\left(\overline{\mathbb{Q}}\right)\cap SL_{k}\left(\mathbb{R}\right)$, where $\overline{\mathbb{Q}}$ stands for the algebraic closure of $\mathbb{Q}$.
For every $v\in g\cdot \mathbb{Z}^{k}$, we have that 
\begin{equation*}
    Nm(v) = \prod_{i=1}^{k} \left\lvert \sum_{j=1}^{k}g_{i,j}z_{j} \right\rvert,
\end{equation*}
for some fixed vector $z\in \mathbb{Z}^{k}$.

\begin{defn}
Let
\begin{equation*}
\mathcal{K}: K_{0}\leq K_{1} \leq \cdots \leq K_{k}
\end{equation*}
be a tower of real number fields such that
\begin{equation*}
\left[K_{j+1}:K_{j}\right]\geq k+1
\end{equation*}
for $j=0,1,\ldots,k-1$.
Denote by $Cl(\mathcal{K})$ the class of matrices in $GL_{k}(\mathbb{R})$ where the elements of the $j$'th column $\xi_{1,j},\ldots,\xi_{k,j}$ belong to the field $K_{j}$ and the algebraic numbers $1,\xi_{1,j},\ldots,\xi_{k,j}$ are linearly independent over $K_{j-1}$ for $j=1,\ldots,k$.
We denote by $SCl(\mathcal{K})$ the natural projection of $Cl(\mathcal{K})$ from $GL_{k}(\mathbb{R})$ to $SL_{k}(\mathbb{R})$.
\end{defn}
\begin{ex}
Assume that $x_0$ is a non-$H$-periodic point in $G/\Gamma$ which has the following realization $x_0=g.SL_{k}(\mathbb{Z})$ for some $g\in SCl(\mathcal{K})$, for some tower of number fields $\mathcal{K}$.
Let $\Nm(y)$ be the following function $\Nm(y)=\prod_{i=1}^{k}\lvert y_{i}\rvert$ defined for any $y\in \mathbb{R}^{k}$.
For a lattice $L\leq \mathbb{R}^k$, define $v(L,\rho)$ to be
\begin{equation*}
v(L,\rho)=\min \left\{ \lvert \Nm(y) \rvert \mid y\in L, 0<\lVert y\rVert<\rho \right\}.
\end{equation*}
By~\cite[Lemma~$5.4$]{Skriganov}, we have for every $\epsilon>0$ and diagonal matrix $a$ the following estimate
\begin{equation*}
v(a. x_{0},\rho) \geq c_{\epsilon}\cdot \left( \rho\cdot a^{-1}_{min}\right)^{-\epsilon},
\end{equation*}
where $a_{min}=\min_{i=1,\ldots,k}\left\{\lvert\left(a\right)_{i,i}\rvert\right\}$ and $c_{\epsilon}$ is an absolute constant.
Choosing $\rho=2^{k}$, and $a=a_{-\log R}$ gives the following estimate for every $\epsilon>0$:
\begin{equation*}
v\left(x_{0},R^{C}\right) \geq c'_{\epsilon}\left(R^{-\epsilon\cdot C\cdot\sigma}\right),
\end{equation*}
where $C$ is some explicit constant, and $R^{-\sigma}$ equals to the smallest diagonal entry appearing in $a_{1}$.
Notice that we have
\begin{equation*}
\lvert \Nm(y) \rvert \leq \lVert y \rVert^{k}_{\infty},
\end{equation*}
and therefore the above estimate controls the length of the shortest vector.
Notice that the estimate given in the proof of~\cite[Lemma~$5.4$]{Skriganov}, uses the subspace theorem by induction argument over the number of linear forms, hence the same estimate (with possibly different constants), holds for any primitive subgroup $\Delta \leq \mathbb{Z}^{k}$ (as $x_0.\Delta$ is a lattice in the corresponding vector space, which also satisfies the assumptions of~\cite[Lemma~$5.4$]{Skriganov}).
Therefore the diophantine condition holds in this algebraic case as well.
\end{ex}

\begin{proof}[Proof of Theorem~\ref{thm:effective-equi} where $\Gamma$ is a non-uniform lattice]
For this proof, we assume that $G$ is a real algebraic group defined over $\mathbb{Q}$ which is realized as a subgroup of $SL_{k}(\mathbb{R})$ by some explicit embedding $i~:~G\to~SL_{k}(\mathbb{R})$.
The only difficulty of applying the proof given above for non-compact homogeneous spaces is hidden in Lemma~\ref{lem:key-lemma}. During the course of the proof of this Lemma, we have assumed that $x',y'$ are $\delta$-close and in particular $\delta$ is much smaller than the injectivity radius at $x'\in G/\Gamma$. In the non-compact setting, the infimum of the  injectivity radius over the whole space shrinks to $0$ as this radius shrinks along the cusps, hence we need to carefully control the injectivity radius of $x'=a_{-\log R}.x$ for various times $R$ in order to satisfy condition~\eqref{eq:delta-condition}.
This is done via the assumption about the diophantine nature of the point $x\in G/\Gamma$.
Let $D$ be any positive number such that 
\begin{equation*}
    D< 2s\cdot\frac{d_{H}}{k\cdot\dim G}.
\end{equation*}
With this choice of $D$, we have that if $y\in X^{1}_{\geq R^{-D}}$, injectivity radius requirement in~\eqref{eq:delta-condition} is satisfied for $y$.

Given $R$, we may write
\begin{equation*}
\begin{split}
B^{H}_{\Theta(R)\cdot R}.x_0 &= a_{\log(\Theta(R)\cdot R)}.B^{H}_{1}.a_{-\log(\Theta(R)\cdot R)}.x_0 \\
&= a_{\log(R)}.B^{H}_{\Theta(R)}a_{-\log(R)}.x_0.
\end{split}
\end{equation*}
By the assumption that the point $x$ satisifies the diophantine condition,
\begin{equation*}
    B^{H}_{\Theta(R)}a_{-\log(R)}.x_0 \cap X^{1}_{\geq R^{-D}} \neq \emptyset.
\end{equation*}
Therefore, by Corollary~\ref{cor:inheritaed-diophantine-cond1}, we have that most of the points in the set  $B^{H}_{\Theta(R)}a_{-\log(R)}.x_0$ are contained in $X^{1}_{\geq R^{-D-\epsilon}}$.
At this point we may and will assume that $\Theta(R)\geq R$.
If so, one may easily verify that
\begin{equation*}
\begin{split}
    \frac{1}{\vol_{H}(B^{H}_{R})}\int_{v\in B^{H}_{R}}\frac{1}{\vol_{H}(B^{H}_{R\cdot \Theta(R)})}\int_{u\in B^{H}_{R\cdot \Theta(R)}}f(v.u.x)dudv \\
    = \frac{1}{\vol_{H}(B^{H}_{R\cdot \Theta(R)})}\int_{u\in B^{H}_{R\cdot \Theta(R)}}f(u.x)du + O_{f}\left(\frac{1}{\Theta(R)}\right),
\end{split}
\end{equation*}
by a Lipschitz estimate.

We may estimate the left hand-side of the equation as follows
\begin{equation}
\begin{split}
   &\frac{1}{\vol_{H}(B^{H}_{R})}\int_{v\in B^{H}_{R}}\frac{1}{\vol_{H}(B^{H}_{R\cdot \Theta(R)})}\int_{u\in B^{H}_{R\cdot \Theta(R)}}f(v.u.x)dudv \\
   &= \frac{1}{\vol_{H}(B^{H}_{R})}\int_{v\in B^{H}_{R}}\frac{1}{\vol_{H}(B^{H}_{\Theta(R)})}\int_{u\in B^{H}_{\Theta(R)}}f(v.a_{\log R}.u.a_{-\log R}.x)dudv \\
   &= \frac{1}{\vol_{H}(B^{H}_{1})}\int_{v\in B^{H}_{1}}\frac{1}{\vol_{H}(B^{H}_{\Theta(R)})}\int_{u\in B^{H}_{\Theta(R)}}f(a_{\log R}.v.u.a_{-\log R}.x)dudv.
\end{split}
\end{equation}
We split the integration region $\left\{u\in B^{H}_{\Theta(R)}\right\}\subset H$ as follows:
$$\left\{u\in B^{H}_{\Theta(R)}\right\} = \mathcal{A}\cup \mathcal{B},$$
where $$\mathcal{A}=\left\{u\in B^{H}_{\Theta(R)} : u.a_{-\log R}.x \in X^{1}_{\geq R^{-D-\epsilon}}\right\},$$ and $$\mathcal{B}=\left\{u\in B^{H}_{\Theta(R)} : u.a_{-\log R}.x \notin X^{1}_{\geq R^{-D-\epsilon}}\right\}.$$
By Corollary~\ref{cor:inheritaed-diophantine-cond1}, we have that
\begin{equation*}
    \frac{\vol_{H}(\mathcal{B})}{\vol_{H}(B^{H}_{\Theta(R)})} = O\left(\Theta(R)^{-\alpha\cdot\epsilon}\right).
\end{equation*}
Therefore we have the estimate
\begin{equation*}
\begin{split}
    &\frac{1}{\vol_{H}(B^{H}_{1})}\int_{v\in B^{H}_{1}}\frac{1}{\vol_{H}(B^{H}_{\Theta(R)})}\int_{u\in B^{H}_{\Theta(R)}}f(a_{\log R}.v.u.a_{-\log R}.x)dudv \\
    &= \frac{1}{\vol_{H}(B^{H}_{1})}\int_{v\in B^{H}_{1}}\frac{1}{\vol_{H}(B^{H}_{\Theta(R)})}\int_{u\in\mathcal{A}}f(a_{\log R}.v.u.a_{-\log R}.x)dudv \\
    &\ + \frac{1}{\vol_{H}(B^{H}_{1})}\int_{v\in B^{H}_{1}}\frac{1}{\vol_{H}(B^{H}_{\Theta(R)})}\int_{u\in\mathcal{B}}f(a_{\log R}.v.u.a_{-\log R}.x)dudv \\
    &= \frac{1}{\vol_{H}(B^{H}_{1})}\int_{v\in B^{H}_{1}}\frac{1}{\vol_{H}(B^{H}_{\Theta(R)})}\int_{u\in\mathcal{A}}f(a_{\log R}.v.u.a_{-\log R}.x)dudv + O_{f}\left(\Theta(R)^{-\alpha\cdot\epsilon}\right).
\end{split}
\end{equation*}
As for estimating the first summand, exchanging the integrals we have
\begin{equation*}
\begin{split}
&\frac{1}{\vol_{H}(B^{H}_{1})}\int_{v\in B^{H}_{1}}\frac{1}{\vol_{H}(B^{H}_{\Theta(R)})}\int_{u\in\mathcal{A}}f(a_{\log R}.v.u.a_{-\log R}.x)dudv \\
&= \frac{1}{\vol_{H}(B^{H}_{\Theta(R)})}\int_{u\in \mathcal{A}}\frac{1}{\vol_{H}(B^{H}_{1})}\int_{v\in B^{H}_{1}}f(a_{\log R}.v.u.a_{-\log R}.x)dvdu.
\end{split}
\end{equation*}
Writing $y=y(u)$ by $y=a_{\log R}.u.a_{-\log R}.x\in B^{H}_{R}.x$ we see that the inner integral gives - 
\begin{equation*}
\frac{1}{\vol_{H}(B^{H}_{1})}\int_{v\in B^{H}_{1}}f(a_{\log R}.v.u.a_{-\log R}.x)dv=\frac{1}{\vol_{H}\left(B^{H}_{R}\right)}\int_{v\in B^{H}_{R}}f(v.y)dv.
\end{equation*}
As $D$ was chosen specifically so that the volume condition in~\eqref{eq:delta-condition} will hold, with respect to the effective mean ergodic theorem for average over $B^{H}_{R}$ achieved in Lemma~\ref{lem:effective-ergodic}, we may estimate $\frac{1}{\vol_{H}\left(B^{H}_{R}\right)}\int_{v\in B^{H}_{R}}f(v.y)dv$ by $C\cdot R^{-s'}\cdot Sob_{K}(f)$ for the same $s'$ which was calculated during the proof of the compact case of the theorem.
Combining all the above averages to one yields:
\begin{equation}
    \frac{1}{\vol_{H}(B^{H}_{R\cdot\Theta(R)})}\int_{u\in B^{H}_{\Theta(R)\cdot R}}f(u.x)du \ll R^{-s'}\cdot Sob_{K}(f)+\Theta(R)^{-1}\cdot\lVert f \rVert_{Lip}+R^{-\alpha\cdot\epsilon}\cdot \lVert f \rVert_{\infty}.
\end{equation}
Assuming that $\Theta(R)$ is bounded by a polynomial of degree $p$ (in $R$), we may enlarge $\Theta$ if necessary and assume $\Theta(R)=O(R^{p})$, and as a result $R\cdot \Theta(R)=O(R^{p+1})$, which gives the following estimate
\begin{equation*}
    \frac{1}{\vol_{H}(B^{H}_{R\cdot\Theta(R)})}\int_{u\in B^{H}_{R^{p+1}}}f(u.x)du = O(R^{-s'})+O(R^{-p})+O(R^{-\alpha\cdot\epsilon}),
\end{equation*}
which upon renormalization gives
\begin{equation}
    \frac{1}{\vol_{H}(B^{H}_{T})}\int_{u\in B^{H}_{T}}f(u.x)du = O(T^{-s'/(p+1)})+O(T^{-p/(p+1)})+O(T^{-\alpha\cdot\epsilon/(p+1)}),
\end{equation}
and in particular:
\begin{equation*}
    \frac{1}{\vol_{H}(B^{H}_{T})}\int_{u\in B^{H}_{T}}f(u.x)du \ll T^{-\gamma}\cdot Sob_{K}(f),
\end{equation*}
for some $\gamma=\gamma(s,p,H)>0$, as $\alpha=\alpha(H)$.

\end{proof}

\begin{rem}
The above proof should be compared to A. Str\"ombergsson's treatment~\cite[Section \S3]{STR13} of the extension of M. Burger's result~\cite{burger1990} into the non-uniform settings.
In particular, his function $r(T)~=~T\cdot~e^{-\dist a_{\log T}(p)}$, which is comparable to $T/\mathcal{Y}_{\Gamma}(pa(T))$ is related to the injectivity radius at $a_{-\log T}.x_0$ in our notation.
The cutting procedure described in the course of his proof of Theorem~$1$ should be thought as a more precise treatment of ours, based on an explicit analysis of the height function.
It would be interesting to generalize his cutting procedure to general rank-$1$ groups, and to extend Burger's integration formula~\cite[Lemma~$1$]{burger1990} for this setting in order to achieve sharper estimates (c.f. \cite[Theorem~$1$]{STR13}), and in-particular to be able to control ``Liouvillian points'' which are not controlled by our method.
Generalizing his cutting procedure in the higher rank cases seems substantially harder, as one would have to consider the various parabolic filtrations, giving rise to different Eisenstein series in the spectral decomposition of $L^{2}(G/\Gamma)$.
%
\end{rem}
\begin{rem}
A quantitative non-divergence result for rank-$1$ spaces (and reducible products of such) can be in principal extracted from the work of C.D. Buenger and C. Zheng~\cite[Theorem~$1.3$]{buenger_zheng_2017}, and our method of proof would work in those cases as well, given a quantitative non-divergence statement.
As the maximum between $(C,\alpha)$-good functions is a $(C,\alpha)$-good function, we may form quantitative non-divergence statement for reducible products of different homogenous spaces easily.
Upon achieving such an estimate, our method should generalize naturally and extend our result to the settings of general semisimple Lie groups.
\end{rem}

\section{Proof of Theorem~\ref{thm:effective-disjointness}}\label{sec:quant-disjointness}
Now we conclude the proof of the quantitative disjointness theorem.
We will not keep track of the Sobolev norm dependence carefully, in general, one may extract such dependence from the calculus of Sobolev norms developed in~\cite[Lemmas~$2.2,8.1$]{venkatesh10}.
In practice, our method of proof transforms the twisted averages to averages of ''additive derivatives'' of the function $f$ (in a way similar to the definition of additive derivatives of a function given in the definition of the Gowers norms of a function), and then one needs only to verify that this ''differentiated function'' lie in the proper Sobolev space for which the polynomial equidstribution holds.
This substantially more complex in the case where the functions are unbounded (see Remark~\ref{rem:sobolev}), but essentially doable along our proof, at least for functions which grow slowly enough at the cusp.
\begin{proof}[Proof of Theorem~\ref{thm:effective-disjointness}]
The proof follows the induction scheme of Theorem~\ref{thm:non-effective} by inducting over the degree of the nilcharacter $\psi$.
Given a homogeneous space $X=G/\Gamma$, $x\in G/\Gamma$ a point for which the orbit under $H\leq G$ equidistributes with a polynomial rate $\gamma_{\text{Equidistribution}}$ for a function with finite Sobolev norm of order $K$ as in Definition~\ref{def:quant-equi} and $s>0$ is a bound for the decay rate of matrix coefficients as in~\eqref{eq:HC-bound}, we define the quantity $\gamma_{\dim N}$ as  
\begin{equation}\label{eq:gamma-n-def}
\gamma_{\dim N}=\frac{1}{\left(2\cdot\dim H +2\right)^{\dim N}}\left(\frac{2M}{2M+1}\right)^{\dim N}\min\left\{\gamma_{\text{Equidistribution}},\frac{s}{d_{H}\cdot(2d_{H}+3s)}\right\}.
\end{equation}
Theorem~\ref{thm:effective-disjointness} amounts to showing that the bound in~\eqref{eq:quant-disjoint} holds for any $\eta<\gamma_{\dim N}$, where $N$ is a nilpotent group for which the nilcharacter $\psi$ is realized as a flow over an associated nilmanifold $N/\Lambda$. 
The base case of the induction, when the nilcharacter $\psi$ is trivial (equivalently, $N=\left\{e\right\}$), clearly follows from the quantitative equidistribution of the averages along the orbit $H.x$ in $G/\Gamma$.
For the induction step, combining the differentiation and approximation properties for the nilcharacter~$\psi$, we may assume that
for a given $\varepsilon>0$ we may replace $\psi(b_1.v.y)\overline{\psi(b_2.v.y)}$ by a combination of $O\left(1/\varepsilon^{\dim H} \right)$ nilcharacters of lesser degree, where the coefficients in the combination are bounded by $O(1)$ at the expense of an error of $O_{\psi}(\varepsilon)$.
Hence by induction hypothesis we deduce that:
\begin{equation}\label{eq:effective-twisted-induction}
\begin{split}
&\left\lvert \frac{1}{\vol_{H}\left(B^{H}_{R}\right)}\int_{v\in B^{H}_{R}} \psi(b_1.v.y)\overline{\psi(b_2.v.y)}\left(f(b_1.v.x)\cdot\overline{f(b_2.v.x)}-\rho_{f,f}(b_1 \cdot b_2 ^{-1})\right)dv \right\rvert \\
&\ll_{f,b_{1},b_{2},\psi} O\left(1/\varepsilon^{\dim H} \right)\vol_{H}\left(B^{H}_{R}\right)^{-\gamma_{\dim N-1}}+O(\varepsilon),
\end{split}
\end{equation}
as the function $\left(f(b_1.v.x)\cdot\overline{f(b_2.v.x)}-\rho_{f,f}(b_1 \cdot b_2 ^{-1})\right)$ is a function of vanishing integral by the definition of the matrix coefficient $\rho_{f,f}(b_{1} \cdot b_{2} ^{-1})$ and where $\gamma_{\dim N-1}$ is given by the formula in~\eqref{eq:gamma-n-def}.
We note that the bound in the above inequality \emph{depend} on $b_{1},b_{2}$ by means of dependence on the Sobolev norm of the function $f(b_{1}.x)\cdot\overline{f(b_{2}.x)}$.
Using the Sobolev inequalities proven in~\cite[Lemma $2.2$]{venkatesh10}, we may resolve this dependence explicitly in the following manner

\begin{equation}\label{eq:effective-twisted-induction-dependence}
\begin{split}
&\left\lvert \frac{1}{\vol_{H}\left(B^{H}_{R}\right)}\int_{v\in B^{H}_{R}} \psi(b_1.v.y)\overline{\psi(b_2.v.y)}\left(f(b_1.v.x)\cdot\overline{f(b_2.v.x)}-\rho_{f,f}(b_1 \cdot b_2 ^{-1})\right)dv \right\rvert \\
&\ll_{f,\psi} \lVert \log(b_{1}) \rVert ^{M}\cdot \lVert \log(b_{2}) \rVert^{M} \cdot \left( O\left(1/\varepsilon^{\dim H} \right)\vol_{H}\left(B^{H}_{R}\right)^{-\gamma_{\dim N-1}}+O(\varepsilon)\right),
\end{split}
\end{equation}
where $M$ is the order of the Sobolev norm being used in the mixing estimate.

Optimizing for $\varepsilon$ we may take $\varepsilon=\vol_{H}\left(B^{H}_{R}\right)^{-\frac{\gamma_{\dim N-1}}{\dim H +1}}$, which amounts to the following bound
\begin{equation}
\label{eq:effective-twisted-induction-dependence-optimized}
\begin{split}
&\left\lvert \frac{1}{\vol_{H}\left(B^{H}_{R}\right)}\int_{v\in B^{H}_{R}} \psi(b_1.v.y)\overline{\psi(b_2.v.y)}\left(f(b_1.v.x)\cdot\overline{f(b_2.v.x)}-\rho_{f,f}(b_1 \cdot b_2 ^{-1})\right)dv \right\rvert \\
&\ll_{f,\psi} \lVert \log(b_{1}) \rVert ^{M}\cdot \lVert \log(b_{2}) \rVert^{M} \cdot \left(\vol_{H}\left(B^{H}_{R}\right)^{-\frac{\gamma_{\dim N -1}}{\dim H +1}}\right).
\end{split}
\end{equation}

Fixing $B=B^{H}_{r}\subset H$ for some $r<R$, we may average the above bound over $b_1, b_2\in B^{H}_{r}$ as follows 
\begin{equation*}
\begin{split}
&\frac{1}{ \vol_{H}\left(B^{H}_{r}\right)^{2}}
\iint_{b_1 , b_2\in B^{H}_{r}}
\biggl\lvert
\frac{1}{\vol_{H}\left(B^{H}_{R}\right)}
\int_{v\in B^{H}_{R}} \psi(b_1.v.y)\overline{\psi(b_2.v.y)}(f(b_1.v.x)\cdot\overline{f(b_2.v.x)} \\ 
&~ \ \ -\rho_{f,f}(b_1 \cdot b_2^{-1}))dv \biggl\rvert db_{1}db_{2} \\
&\ll_{f,\psi} r^{2M\cdot d_{H}}\cdot\vol_{H}\left(B^{H}_{R}\right)^{\frac{-\gamma_{\dim N-1}}{\dim H +1}},
\end{split}
\end{equation*}
where we implicitly used that fact that
\begin{equation*}
\begin{split}
    \frac{1}{\vol_{H}(B^{H}_{r})}\int_{B^{H}_{r}}\lVert \log(b)\rVert^{M} &\ll \max_{Lie(B^{H}_{r})}\left\{\lVert v \rVert^{M} \right\} \\
    &\ll r^{M\cdot d_{H}},
\end{split}
\end{equation*}
which follows from the definition of $B^{H}_{r}$ as the exponential image of the related subsets in $Lie(H)$ and from the definition of $d_{H}$ as the trace of the adjoint action of $a_{log r}$ on $Lie(H)$.

Using the bound for decay of matrix coefficients as given in \eqref{eq:HC-bound}, we may estimate the sum of matrix coefficients in the following manner -
\begin{equation*}
\begin{split}
\frac{1}{\vol_{H}\left(B^{H}_{r}\right)^{2}}\iint_{b_{1},b_{2}\in B}\lvert\rho_{f,f}\left(b_{1} \cdot b_{2}^{-1}\right) \rvert &\ll \frac{\vol(B_{t}^{H})}{\vol_{H}\left(B^{H}_{r}\right)}\\
&\ + \frac{1}{\vol_{H}\left(B^{H}_{r}\right)^{2}}\iint_{b_{1},b_{2}\in B^{H}_{r}: b_{1} b_{2}^{-1} \notin B^{H}_{t}} \lvert\rho_{f,f}\left(b_{1} \cdot b_{2}^{-1}\right) \rvert \\
&\ll \frac{t^{d_H}}{\vol_{H}\left(B^{H}_{r}\right)} + t^{-s},
\end{split}
\end{equation*}
where $s$ is a bound for the polynomial decay rate of the matrix coefficients.
Optimizing $t$ we get error of $\vol_{H}\left(B^{H}_{r}\right)^{-s/(d_{H}+s)}$ for the sum of matrix coefficients.

Hence we may bound the summands in~\eqref{eq:main-non-eff-estimate} by
\begin{equation}
\label{eq:mid-computation}
\begin{split}
r^{2M\cdot d_{H}}\vol_{H}\left(B^{H}_{R}\right)^{-\frac{\gamma_{\dim N-1}}{\dim H +1}} + \vol_{H}\left(B^{H}_{r}\right)^{-s/(d_{H}+s)} \\ + O_{f}\left(\vol_{H}\left(B^{H}_{r}\right) \cdot \frac{\sup_{b\in B^{H}_{r}}\vol_{H}\left(b.B^{H}_{R}\triangle B^{H}_{R}\right)}{\vol_{H}\left(B^{H}_{R}\right)} \right).
\end{split}
\end{equation}
The third term is clearly dominated by the first one, hence optimizing between the first and second shows that one may pick
$$ r= R^{\gamma_{\dim  N-1}\cdot\left(\frac{1}{\dim H +1}\right)\cdot\left(\frac{1}{2M+\frac{s}{d_{H}+s}} \right)}, $$
in order to get an effective bound in inequality~\eqref{eq:mid-computation}.
For example choosing  
$ r= R^{\gamma_{\dim  N-1}\cdot\left(\frac{1}{\dim H +1}\right)\cdot\left(\frac{1}{2M+1} \right)}, $ 
reflects as getting the following estimate for each term in inequality~\eqref{eq:mid-computation}
\begin{equation}\label{eq:end-computation}
O_{f}\left(\vol_{H}\left(B^{H}_{R} \right)^{-\gamma_{\dim N -1}\cdot \left(\frac{1}{\dim H +1} \right)\cdot\left(\frac{2M}{2M+1} \right) }\right).    
\end{equation} 
Using the induction hypothesis, we see that both of the expressions in~\eqref{eq:end-computation} are bounded by $2\cdot \gamma_{\dim N}$ with $\gamma_{\dim N}$ as defined in~\eqref{eq:gamma-n-def}.
Upon taking square root we get:
\begin{equation}
\left\lvert\frac{1}{\vol_{H}\left(B^{H}_{R}\right)}\int_{v \in B^{H}_{R}}\psi(v.y)f(v.x)dv \right\rvert \ll_{f}  \vol_{H}\left(B^{H}_{R}\right)^{-\gamma_{\dim N}},
\label{eq:full-effective-twisted-1}
\end{equation}
proving the estimate in~\eqref{eq:quant-disjoint}.
\end{proof}

We end this section by proving an effective \emph{discrete} version of Theorem~\ref{thm:effective-disjointness}, as such results are of interest in some applications and moreover, the idea of using the disjointness as a method to apply certain summation process by studying approperiate spectral kernels have been used by most the papers studying sparse equidistribution up to date~\cite{venkatesh10, Flaminio2016,ubis2016effective,mcadam} and will be used in subsequent paper of the author towards applications in sparse equdistribution problems~\cite{Katz19}.
We prove the theorem only in the case of \emph{abelian horospherical group}, where the samplings from $H$ are drawn along an abelian subgroup isomorphic to $\mathbb{Z}^{\dim H}$, by using Venkatesh's method.
The case where $H$ is a general nilpotent group does not follow from our proof (as it relies on abelian Fourier analysis). In particular the analysis of such case is closely related to questions about effective equidistribution of discrete nilpotent actions on nilmanifolds and their diophantine behavior, which are not considered explicitly by Green-Tao. We hope to explore such questions in future work.
Examples for such abeliean horospherical groups are minimal horospherical subgroups of $\prod_{i=1}^{n}SL_{2}(\mathbb{R}), SL_{n}(\mathbb{R}),$ and $SO(n,1)(\mathbb{R})$.

We fix $H\leq G$ an abelian horospherical subgroup which is isomorphic to the group $\mathbb{R}^{\dim H}$.
The following is a specialization of Theorem~\ref{thm:effective-disjointness}:
\begin{cor}\label{cor:twisted-effective}
	Assume that $f:G/\Gamma\to \mathbb{R}$ is a smooth and bounded function with $\int_{X}fd\mu = 0$ and finite Sobolev norm of order $K$ and $\psi \in \widehat{H}$ is a character, then we have the following bound for any $H$-generic point $x \in X$ which is of polynomial equidistribution rate $\gamma_{\text{Equidistribution}}$ with respect to functions with $K$-order Sobolev norm -
	\begin{equation}
	\left\lvert \frac{1}{\vol_{H}(B^{H}_{R})}\int_{u\in B^{H}_{R}}\psi(u)f(u . x)d\mu \right\rvert \ll_{f} R^{-\eta\cdot d_H},
	\end{equation}
	for some $\eta=\eta(\Gamma,x , \gamma_{\text{Equidistribution}})$.
\end{cor}
We show how to deduce from such result a discrete quantitative equidistribution result:
\begin{thm}\label{thm:discrete-equi}
In the same settings as above we get the following \emph{discrete equidistribution} result:
\begin{equation}
\left\lvert \frac{1}{\sharp\left\{\mathbb{Z}^{\dim H}\cap B^{H}_{R} \right\}} \sum_{v\in \mathbb{Z}^{\dim H}\cap B^{H}_{R}}f(u_{v}.x) \right\rvert \ll_{f} R^{-\eta'\cdot d_H}.
\end{equation}
for some $\eta'=\eta'(\Gamma,x,\gamma_{\text{Equidistribution}})$.
\end{thm}
\begin{proof}
	The proof is an immediate generalization of~\cite[Theorem~$3.1$]{venkatesh10}.
	Fix $g_{1}(v)$ to be a smooth bump function of total mass equal to $1$, supported in the ball of radius $1$ around $0$.
	Define the family of bump functions, $g_{\delta}(v)$ as follows:
	\begin{equation*}
	    g_{\delta}(v) = \delta^{-\dim H}g_{1}(\delta^{-1}v).
	\end{equation*}
	For each $\delta>0$ the function $g_{\delta}(v)$ is a smooth bump function on $\mathbb{R}^{\dim H}$ supported inside a $\delta$-neighborhood of $0$.
	For $\lambda \in \widehat{\mathbb{R}^{\dim H}}$ write $\hat{g_{\delta}}(\lambda)~=~\int_{\mathbb{R}^{\dim H}}\exp\left(-2\pi i \left<\lambda,v\right>\right)g_{\delta}(v)dv$.
	As $g_{\delta}$ is smooth, its Fourier coefficients decay rapidly and in-particular they are summable, and satisfy the following inequality
	\begin{equation*}
	\sum_{k\in\mathbb{Z}^{\dim H}} \lvert \hat{g_{\delta}}(k) \rvert \ll \delta^{-\dim H}.
	\end{equation*}
	Define the function $F_{R}:\mathbb{R}^{\dim H} \to \mathbb{C}$ by
	\begin{equation*}
	    F_{R}(v)= \begin{cases} \frac{1}{\vol_{H}(B^{H}_{R})}f(v.x) &\ \lVert v \rVert \leq R \\
	    0 &\ \text{otherwise}
	    \end{cases}.
	\end{equation*}
	We have that $F_{R}\in L^{1}(\mathbb{R}^{\dim H})$.
	Moreover, we have that $$\lVert F_{R}\star g_{\delta}-F_{R}\rVert_{\infty} \ll_{f} \delta,$$ by a Lipschitz estimate.
	Using Poisson summation for $F_{R}\star g_{\delta}$ we have
	\begin{equation*}
	\begin{split}
	    \sum_{n\in \mathbb{Z}^{\dim H} \cap B^{H}_{R}}F_{R}\star g_{\delta}(n) &= \sum_{n\in \mathbb{Z}^{\dim H}} \widehat{F_{R}\star g_{\delta}}(n) \\
	    &= \sum_{n\in \mathbb{Z}^{\dim H}} \hat{F_{R}}(n)\hat{g_{\delta}}(n). 
	\end{split}
	\end{equation*}
	Using Corollary~\ref{cor:twisted-effective} we have $\lvert\hat{F_{R}}(n)\rvert \ll_{f} R^{-\eta}$ uniformly in $n$, so by using the  summability of the Fourier coefficients of $g_{\delta}(v)$ we have that
	\begin{equation*}
	    \left\lvert \sum_{n\in \mathbb{Z}^{\dim H} \cap B^{H}_{R}}F_{R}\star g_{\delta}(n) \right\rvert \ll_{f} \delta^{-\dim H}\cdot R^{-\eta}, 
	\end{equation*}
	From which we deduce that
	\begin{equation*}
	    \left\lvert \sum_{n\in \mathbb{Z}^{\dim H}\cap B^{H}_{R}}F_{R}(n) \right\rvert \ll_{f} \delta+\delta^{-\dim H}\cdot R^{-\eta}.
	\end{equation*}
	Optimizing for $\delta$ we have that $\delta=R^{-\eta/(\dim H+1)}$.
	As we have the following estimate for the main term of the number of lattice points in a ball
	$\lim_{R\to \infty}\frac{\sharp\left\{\mathbb{Z}^{\dim H}\cap B^{H}_{R} \right\}}{\vol_{H}(B^{H}_{R})}=1$, the theorem follows for every $R\gg 0$.

\end{proof}
\begin{cor}\label{cor:discrete-equi-twisted}
In the same settings as above we get the following \emph{discrete equidistribution} result for any character $\psi \in \widehat{H}$:
\begin{equation}
\left\lvert \frac{1}{\sharp\left\{\mathbb{Z}^{\dim H}\cap B^{H}_{R} \right\}} \sum_{v\in \mathbb{Z}^{\dim H}\cap B^{H}_{R}}\psi(v.y)f(v.x) \right\rvert \ll_{f} R^{-\eta'\cdot d_{H}}
\end{equation}
\end{cor}
\begin{proof}
Following the Poisson summation argument in Theorem~\ref{thm:discrete-equi}, we see that twisting the average by $\psi$ only affects the characters by translation on the Fourier side, as the proof of Corollary~\ref{cor:twisted-effective} is uniform in $\psi$, the computations follow through.
\end{proof}

\bibliographystyle{plain}
\bibliography{disjointness-nilflow-horocyclic}

\end{document}